\documentclass[10pt,a4paper]{article}

\usepackage[a4paper,left=2.5cm,right=2.5cm,
top=2.5cm,bottom=2.5cm]{geometry}

\usepackage{amsmath,amssymb,amsfonts,amsthm}
\usepackage{mathrsfs}
\usepackage{graphicx}
\usepackage{booktabs}
\usepackage{multirow}
\usepackage{xcolor}
\usepackage{enumerate}
\usepackage{appendix}
\usepackage{hyperref}

\newtheorem{theorem}{Theorem}[section]

\newtheorem{lemma}[theorem]{Lemma}

\theoremstyle{definition}

\newtheorem{example}[theorem]{Example}

\theoremstyle{remark}

\numberwithin{equation}{section}

\title{\bfseries
Ricci curvature bounds for Statistical Submersions
}

\author{
Ravindra Singh
\\[2ex]
Department of Mathematics, Banaras Hindu University, Varanasi 221005, India\\
E-mail: khanelrs@bhu.ac.in\\
ORCID: 0009-0009-1270-3831
}

\date{}

\begin{document}

\maketitle

\begin{abstract}
In this paper, we study Ricci curvature bounds for statistical submersions
from a direct Ricci-curvature perspective. Although Chen--Ricci inequalities
for statistical submersions have already been established, a complementary
lower estimate is needed to obtain a two-sided description of the Ricci
curvature. Motivated by this observation, we combine the Ricci curvature
relations of a statistical submersion with Hineva's algebraic inequality.
We first obtain a Chen--Ricci upper estimate and a Hineva-type lower
estimate along the vertical distribution, expressed in terms of the
intrinsic Ricci curvature of the fibres and the fundamental tensors
$T$ and $T^{*}$. The two estimates together provide upper and lower
bounds for the vertical Ricci curvature. We then derive a Chen--Ricci
inequality along the horizontal distribution involving the fundamental
tensors $A$ and $A^{*}$. By combining the vertical and horizontal
curvature relations, we further obtain corresponding Chen--Ricci and
Hineva-type estimates for the mixed distribution. The equality conditions
are characterized in terms of the components of the fundamental tensors
and their duals. Explicit examples are given to illustrate both equality
and strict inequality cases. Thus, the paper provides a unified
Ricci-curvature approach to statistical submersions and, in particular,
gives the first Hineva-type lower Ricci curvature estimates for statistical
submersions, leading to two-sided Ricci curvature bounds.
\end{abstract}



\section{Introduction}

The study of curvature inequalities is an important part of differential
geometry, particularly in understanding the relationship between
intrinsic curvature and the geometric structure of a manifold. Among the
various curvature quantities, Ricci curvature plays a central role since
it describes the curvature in a prescribed tangent direction. This makes
Ricci curvature inequalities particularly useful in the study of
submanifolds, submersions, and related geometric structures.

Statistical geometry provides a natural extension of Riemannian geometry
by replacing the single Levi--Civita connection with a pair of mutually
dual torsion-free affine connections. More precisely, a statistical
manifold is a triple $(M,\nabla,g)$ satisfying
\[
Xg(Y,Z)
=
g(\nabla_XY,Z)+g(Y,\nabla_X^{*}Z).
\]
The duality between $\nabla$ and $\nabla^{*}$ introduces additional
geometric information and makes statistical manifolds an important
framework in information geometry and statistical inference; see
\cite{Noguchi1992,AmariNagaoka2000}.

The geometry of Riemannian submersions was developed systematically by
O'Neill \cite{ONeill1966} and Gray \cite{Gray1967}. A Riemannian
submersion
\[
F:(M,g_{1})\longrightarrow(N,g_{2})
\]
naturally decomposes the tangent bundle into the vertical and horizontal
distributions,
\[
TM=\mathcal V\oplus\mathcal H,
\qquad
\mathcal V=\ker F_{*},
\qquad
\mathcal H=(\ker F_{*})^{\perp}.
\]
The interaction between these two distributions is described by the
fundamental tensors $T$ and $A$. In recent years, various general
curvature inequalities for Riemannian submersions and Riemannian maps
have been established, including generalized Chen-type inequalities,
Chen--Ricci inequalities, and Casorati curvature inequalities; see,
for example,
\cite{Singh_26_GCI,Singh_26_RCB,Singh_26_CSM,Singh_26_GCF,Singh_26_GCR,Singh_Meena_Meena_26_JMAA,Singh_26_GCFRM}.
These developments provide a natural geometric background for
investigating analogous curvature inequalities in more general
settings. The statistical counterpart of this
theory was developed through affine submersions by Abe and Hasegawa
\cite{AbeHasegawa2001} and was subsequently developed further by Takano
\cite{Takano2004}. Various classes of statistical submersions associated
with additional geometric structures have since been investigated; see
\cite{AytimurOzgur2019,VilcuVilcu2015,Vilcu2021,
SiddiquiAlghamdi2025}.

Curvature inequalities for statistical submersions have also received
attention. In particular, Chen--Ricci inequalities provide upper
estimates for Ricci curvature in terms of geometric quantities associated
with a submersion. In the statistical setting, Siddiqui, Chen and
Siddiqi \cite{SiddiquiChenSiddiqi2021} established Chen--Ricci
inequalities for statistical submersions and obtained, in addition, a
Chen-type inequality involving the $\delta(2,2)$-invariant. Thus, the
existence of Chen--Ricci inequalities in the setting of statistical
submersions is already known.

A natural question is whether these upper estimates can be complemented
by corresponding lower estimates. Such a complementary viewpoint is
provided by Hineva's approach \cite{Hineva2008}, which gives algebraic
lower bounds for Ricci curvature and characterizes their equality cases.
This leads to a natural possibility in the statistical submersion setting:
rather than considering only an upper Ricci curvature estimate, one can
seek a pair of estimates which controls the Ricci curvature from both
sides. The presence of the dual connections makes this problem especially
interesting, since the fundamental tensors occur in dual pairs $(T,T^{*})$ and $(A,A^{*})$.

Motivated by this observation, in the present paper we develop a direct
Ricci-curvature approach to curvature inequalities for statistical
submersions. The advantage of this approach is that the estimates are
obtained directly from the Ricci curvature relations associated with the
vertical, horizontal, and mixed distributions. We first consider the
vertical distribution. By using the curvature relation for vertical
vectors, we obtain a Chen--Ricci upper estimate involving the intrinsic
Ricci curvature of the fibres and the traces of $T$ and $T^{*}$. We then
apply Hineva's algebraic inequality to the same Ricci curvature relation
and obtain a corresponding lower estimate. Consequently, the vertical
Ricci curvature is bounded from above and below within the same
framework.

The lower estimate obtained in this way gives a Hineva-type inequality
for statistical submersions. To the best of our knowledge, this provides
the first Hineva-type lower Ricci curvature inequality in the setting of
statistical submersions. Thus, the combination of the Chen--Ricci and
Hineva estimates gives a two-sided Ricci curvature bound, which is the
main motivation for the present approach.

The same geometric viewpoint naturally leads to the horizontal
distribution. Using the horizontal curvature relation and the
fundamental tensors $A$ and $A^{*}$, we derive a Chen--Ricci inequality
for the horizontal Ricci curvature. Finally, the vertical and horizontal
curvature relations are combined to study the mixed distribution. This
yields corresponding Chen--Ricci and Hineva-type estimates for the mixed
Ricci curvature and completes the treatment of the three natural
distributions associated with a statistical submersion.

An important part of the study concerns the equality cases. The equality
conditions are expressed in terms of the components of the fundamental
tensors and their duals, and hence describe the geometric configurations
under which the corresponding Ricci curvature bounds are attained. We
also provide explicit examples in which equality holds as well as
examples in which the inequalities are strict. These examples show that
the obtained bounds are non-trivial and that the equality conditions are
genuinely realizable.

Therefore, the main contribution of this work is not the introduction of
the Chen--Ricci inequality itself for statistical submersions, which is
already known, but rather the use of a direct Ricci-curvature approach to
obtain a complementary Hineva-type lower estimate and, consequently, a
two-sided description of Ricci curvature. The vertical, horizontal, and
mixed distributions are treated within a common framework, providing a
basis for further curvature inequalities in statistical submersion
geometry.

The paper is organized as follows. Section~2 recalls the basic definitions,
fundamental tensors, curvature relations, and Ricci curvature quantities
needed in the sequel. In Section~3, we establish the Chen--Ricci and
Hineva inequalities along the vertical distribution. Section~4 is
devoted to the Chen--Ricci inequality along the horizontal distribution.
In Section~5, we derive the corresponding Chen--Ricci and Hineva-type
inequalities along the mixed distribution and discuss their equality
conditions. Finally, examples are presented to illustrate the obtained
estimates and their equality and strictness cases.

\section{Preliminaries}

In this section, we recall the basic notions and curvature formulas for
statistical manifolds and statistical submersions that will be used
throughout the paper. We first recall the decomposition associated with a
Riemannian submersion and then introduce the fundamental tensors associated
with the statistical connections and their duals. We also define the
component notation, trace terms, mean curvature vectors, divergence terms,
Ricci curvatures, scalar curvatures, and the algebraic inequality required
in the sequel.

\subsection{Statistical manifolds and the difference tensor}

Let $(M,g)$ be a Riemannian manifold and let $\nabla$ be a torsion-free
affine connection on $M$. The triple $(M,\nabla,g)$ is called a
\emph{statistical manifold} if there exists a torsion-free affine
connection $\nabla^{*}$, called the dual connection of $\nabla$ with
respect to $g$, such that
\begin{equation}
Xg(Y,Z)
=
g(\nabla_XY,Z)+g(Y,\nabla^{*}_{X}Z)
\label{eq-statistical-manifold}
\end{equation}
for all vector fields $X,Y,Z$ on $M$.

Let $\nabla^{0}$ denote the Levi--Civita connection of the Riemannian
metric $g$. The duality relation in \eqref{eq-statistical-manifold} implies
\begin{equation}
\nabla^{0}
=
\frac{1}{2}\left(\nabla+\nabla^{*}\right).
\label{eq-LC-stat}
\end{equation}
We introduce the difference tensor $K$ by
\begin{equation}
K_{X}Y
=
\nabla_XY-\nabla^{0}_{X}Y
=
-\left(\nabla^{*}_{X}Y-\nabla^{0}_{X}Y\right).
\label{eq-difference-K}
\end{equation}
Consequently,
\begin{equation}
\nabla_XY
=
\nabla^{0}_{X}Y+K_XY,
\qquad
\nabla^{*}_{X}Y
=
\nabla^{0}_{X}Y-K_XY.
\label{eq-connections-K}
\end{equation}
The tensor $K$ is symmetric in $X$ and $Y$ and satisfies
\begin{equation}
g(K_XY,Z)=g(Y,K_XZ).
\label{eq-K-symmetry}
\end{equation}
The tensor $K$ will also be used in the examples to construct explicit
statistical structures.

\subsection{Riemannian submersions and statistical submersions}

Let
\[
F:(M,g_1)\longrightarrow (N,g_2)
\]
be a Riemannian submersion. The tangent bundle of $M$ admits the
orthogonal decomposition
\begin{equation}
TM=\mathcal V\oplus\mathcal H,
\label{eq-vertical-horizontal}
\end{equation}
where
\[
\mathcal V=\ker F_{*},
\qquad
\mathcal H=(\ker F_{*})^{\perp}.
\]
The distributions $\mathcal V$ and $\mathcal H$ are called the
\emph{vertical} and \emph{horizontal} distributions, respectively. We put
\[
r=\dim\mathcal V,
\qquad
s=\dim\mathcal H.
\]

Let $\nabla^{0}$ be the Levi--Civita connection of $(M,g_1)$. The
fundamental tensors of the Riemannian submersion are given by
\begin{equation}
T^{0}_{E}F
=
\mathcal V\nabla^{0}_{\mathcal VE}\mathcal HF
+
\mathcal H\nabla^{0}_{\mathcal VE}\mathcal VF,
\label{eq-T0}
\end{equation}
and
\begin{equation}
A^{0}_{E}F
=
\mathcal V\nabla^{0}_{\mathcal HE}\mathcal HF
+
\mathcal H\nabla^{0}_{\mathcal HE}\mathcal VF,
\label{eq-A0}
\end{equation}
for all vector fields $E,F$ on $M$. In particular, for
$U,V\in\Gamma(\mathcal V)$ and $X,Y\in\Gamma(\mathcal H)$,
\begin{equation}
T^{0}_{U}V=T^{0}_{V}U,
\qquad
A^{0}_{X}Y=-A^{0}_{Y}X
=
\frac12\mathcal V[X,Y].
\label{eq-T0-A0-properties}
\end{equation}

Now let $(M,\nabla,g_1)$ and $(N,\nabla^{2},g_2)$ be statistical
manifolds. The Riemannian submersion
\[
F:(M,\nabla,g_1)\longrightarrow(N,\nabla^{2},g_2)
\]
is called a \emph{statistical submersion} if, for any basic horizontal
vector fields $X,Y$,
\begin{equation}
F_{*}(\nabla_XY)
=
\nabla^{2}_{F_{*}X}F_{*}Y.
\label{eq-stat-submersion}
\end{equation}

The affine connections induced by $\nabla$ and $\nabla^{*}$ on the fibres
are defined by
\begin{equation}
\nabla^{\mathcal V}_{U}V
=
\mathcal V\nabla_UV,
\qquad
\nabla^{*\mathcal V}_{U}V
=
\mathcal V\nabla^{*}_{U}V.
\label{eq-induced-fibre-connections}
\end{equation}

The fundamental tensors associated with $\nabla$ and $\nabla^{*}$ are
defined by
\begin{equation}
T_EF
=
\mathcal V\nabla_{\mathcal VE}\mathcal HF
+
\mathcal H\nabla_{\mathcal VE}\mathcal VF,
\label{eq-T}
\end{equation}
\begin{equation}
T^{*}_EF
=
\mathcal V\nabla^{*}_{\mathcal VE}\mathcal HF
+
\mathcal H\nabla^{*}_{\mathcal VE}\mathcal VF,
\label{eq-Tstar}
\end{equation}
and
\begin{equation}
A_EF
=
\mathcal V\nabla_{\mathcal HE}\mathcal HF
+
\mathcal H\nabla_{\mathcal HE}\mathcal VF,
\label{eq-A}
\end{equation}
\begin{equation}
A^{*}_EF
=
\mathcal V\nabla^{*}_{\mathcal HE}\mathcal HF
+
\mathcal H\nabla^{*}_{\mathcal HE}\mathcal VF.
\label{eq-Astar}
\end{equation}

For $U,V\in\Gamma(\mathcal V)$ and
$X,Y\in\Gamma(\mathcal H)$, the duality relation gives
\begin{equation}
g_1(T_UV,X)
=
-g_1(V,T^{*}_UX),
\label{eq-dual-T}
\end{equation}
and
\begin{equation}
g_1(A_XY,V)
=
-g_1(Y,A^{*}_XV).
\label{eq-dual-A}
\end{equation}
Moreover,
\begin{equation}
A_XY=-A^{*}_YX.
\label{eq-A-Astar}
\end{equation}

The covariant derivatives of vertical and horizontal vector fields can be
decomposed as
\begin{equation}
\nabla_UV
=
T_UV+\nabla^{\mathcal V}_UV,
\qquad
\nabla^{*}_UV
=
T^{*}_UV+\nabla^{*\mathcal V}_UV,
\label{eq-decomp-1}
\end{equation}
\begin{equation}
\nabla_UX
=
T_UX+\mathcal H\nabla_UX,
\qquad
\nabla^{*}_UX
=
T^{*}_UX+\mathcal H\nabla^{*}_UX,
\label{eq-decomp-2}
\end{equation}
\begin{equation}
\nabla_XU
=
A_XU+\mathcal V\nabla_XU,
\qquad
\nabla^{*}_XU
=
A^{*}_XU+\mathcal V\nabla^{*}_XU,
\label{eq-decomp-3}
\end{equation}
and
\begin{equation}
\nabla_XY
=
\mathcal H\nabla_XY+A_XY,
\qquad
\nabla^{*}_XY
=
\mathcal H\nabla^{*}_XY+A^{*}_XY.
\label{eq-decomp-4}
\end{equation}

\subsection{Components, traces and mean curvature vectors}

Let
\[
\left\{
V_1,\ldots,V_r,h_1,\ldots,h_s
\right\}
\]
be a local orthonormal frame adapted to the decomposition
\eqref{eq-vertical-horizontal}, where
\[
\{V_1,\ldots,V_r\}
\]
is a local orthonormal frame of $\mathcal V$ and
\[
\{h_1,\ldots,h_s\}
\]
is a local orthonormal frame of $\mathcal H$.

We define the components of $T,T^{*},A,A^{*}$ by
\begin{equation}
T_{ij}^{\alpha}
=
g_1(T_{V_i}V_j,h_\alpha),
\qquad
T_{ij}^{*\alpha}
=
g_1(T^{*}_{V_i}V_j,h_\alpha),
\label{eq-T-components}
\end{equation}
where
\[
1\leq i,j\leq r,
\qquad
1\leq\alpha\leq s,
\]
and
\begin{equation}
A_{ij}^{\alpha}
=
g_1(A_{h_i}h_j,V_\alpha),
\qquad
A_{ij}^{*\alpha}
=
g_1(A^{*}_{h_i}h_j,V_\alpha),
\label{eq-A-components}
\end{equation}
where
\[
1\leq i,j\leq s,
\qquad
1\leq\alpha\leq r.
\]

The Levi--Civita tensor $T^{0}$ satisfies
\begin{equation}
2T^{\alpha 0}_{ij}
=
T_{ij}^{\alpha}+T_{ij}^{*\alpha}.
\label{eq-T0-components}
\end{equation}

The trace vector fields of $T$ and $T^{*}$ are defined by
\begin{equation}
N=\operatorname{trace}T
=
\sum_{i=1}^{r}T_{V_i}V_i,
\qquad
N^{*}=\operatorname{trace}T^{*}
=
\sum_{i=1}^{r}T^{*}_{V_i}V_i.
\label{eq-trace-T}
\end{equation}
Consequently, the mean curvature vector fields of the fibres with respect
to $\nabla$ and $\nabla^{*}$ are
\begin{equation}
H=\frac{1}{r}N,
\qquad
H^{*}=\frac{1}{r}N^{*}.
\label{eq-mean-curvature}
\end{equation}
In particular, the fibres are minimal with respect to $\nabla$ if and
only if $N=0$, and minimal with respect to $\nabla^{*}$ if and only if
$N^{*}=0$.

Similarly, define the horizontal trace vector fields
\begin{equation}
\sigma
=
\operatorname{trace}A
=
\sum_{i=1}^{s}A_{h_i}h_i,
\qquad
\sigma^{*}
=
\operatorname{trace}A^{*}
=
\sum_{i=1}^{s}A^{*}_{h_i}h_i.
\label{eq-sigma}
\end{equation}

The squared norm of $\operatorname{trace}A$ is
\begin{equation}
\left\|\operatorname{trace}A\right\|^{2}
=
\|\sigma\|^{2}
=
\sum_{\alpha=1}^{r}
\left(
\sum_{j=1}^{s}A_{jj}^{\alpha}
\right)^{2},
\label{eq-nrm-A}
\end{equation}
and similarly
\begin{equation}
\left\|\operatorname{trace}A^{*}\right\|^{2}
=
\|\sigma^{*}\|^{2}
=
\sum_{\alpha=1}^{r}
\left(
\sum_{j=1}^{s}A_{jj}^{*\alpha}
\right)^{2}.
\label{eq-nrm-Astar}
\end{equation}

For the component matrices of $T$ and $T^{*}$, we use
\begin{equation}
\operatorname{trace}(T^{\alpha})
=
\sum_{i=1}^{r}T_{ii}^{\alpha},
\qquad
\|T\|^{2}
=
\sum_{\alpha=1}^{s}
\sum_{i,j=1}^{r}
(T_{ij}^{\alpha})^{2},
\label{eq-T-norm}
\end{equation}
and
\begin{equation}
\operatorname{trace}(T^{*\alpha})
=
\sum_{i=1}^{r}T_{ii}^{*\alpha},
\qquad
\|T^{*}\|^{2}
=
\sum_{\alpha=1}^{s}
\sum_{i,j=1}^{r}
(T_{ij}^{*\alpha})^{2}.
\label{eq-Tstar-norm}
\end{equation}

Likewise, we put
\begin{equation}
\|A\|^{2}
=
\sum_{\alpha=1}^{r}
\sum_{i,j=1}^{s}
(A_{ij}^{\alpha})^{2},
\label{eq-A-norm}
\end{equation}
\begin{equation}
\|A^{*}\|^{2}
=
\sum_{\alpha=1}^{r}
\sum_{i,j=1}^{s}
(A_{ij}^{*\alpha})^{2},
\label{eq-Astar-norm}
\end{equation}
and
\begin{equation}
g_1(A,A^{*})
=
\sum_{\alpha=1}^{r}
\sum_{i,j=1}^{s}
A_{ij}^{\alpha}A_{ij}^{*\alpha}.
\label{eq-A-Astar-inner}
\end{equation}

\subsection{Divergence terms}

The mixed curvature relations used later contain divergence terms associated
with the trace vector fields $N,N^{*},\sigma,\sigma^{*}$. We define
\begin{equation}
\delta(N)
=
-\sum_{i=1}^{s}
g_1(\nabla_{h_i}N,h_i),
\qquad
\delta^{*}(N^{*})
=
-\sum_{i=1}^{s}
g_1(\nabla^{*}_{h_i}N^{*},h_i),
\label{eq-delta-N}
\end{equation}
and
\begin{equation}
\delta(\sigma)
=
-\sum_{i=1}^{r}
g_1(\nabla_{V_i}\sigma,V_i),
\qquad
\delta^{*}(\sigma^{*})
=
-\sum_{i=1}^{r}
g_1(\nabla^{*}_{V_i}\sigma^{*},V_i).
\label{eq-delta-sigma}
\end{equation}

These quantities measure the divergence of the trace parts of the
fundamental tensors with respect to the corresponding affine connections.
They naturally occur in the scalar-curvature decomposition used for the
mixed distribution.

\subsection{Curvature relations}

Let $R^{M}$ and $R^{\mathcal V}$ denote the curvature tensors associated
with $\nabla$ on $M$ and with the induced connection
$\nabla^{\mathcal V}$ on the fibres, respectively.

For vertical vector fields
$F_1,F_2,F_3,F_4\in\Gamma(\mathcal V)$, the statistical Gauss equation is
\begin{equation}
\begin{aligned}
R^{M}(F_1,F_2,F_3,F_4)
={}&
R^{\mathcal V}(F_1,F_2,F_3,F_4)\\
&+
g_1(T_{F_1}F_3,T^{*}_{F_2}F_4)
-
g_1(T_{F_2}F_3,T^{*}_{F_1}F_4).
\end{aligned}
\label{eq-Gauss-stat}
\end{equation}

For horizontal vector fields
$U_1,U_2,U_3,U_4\in\Gamma(\mathcal H)$, we have
\begin{align}
R^{M}(U_1,U_2,U_3,U_4)
={}&
R^{\mathcal H}(U_1,U_2,U_3,U_4)
-
g_1(A_{U_2}U_3,A^{*}_{U_1}U_4)
\nonumber\\
&+
g_1\left(
(A_{U_1}+A^{*}_{U_1})U_2,
A^{*}_{U_3}U_4
\right)
\nonumber\\
&+
g_1(A_{U_1}U_3,A^{*}_{U_2}U_4).
\label{eq-Gauss-hor}
\end{align}

For mixed vector fields
$U_1,U_2\in\Gamma(\mathcal H)$ and
$F_1,F_2\in\Gamma(\mathcal V)$, we have
\begin{equation}
R^{M}(U_1,F_1,F_2,U_2)
=
g_1(A_{U_1}F_1,A^{*}_{U_2}F_2)
-
g_1(T_{F_1}U_1,T^{*}_{F_2}U_2).
\label{eq-Gauss-mix}
\end{equation}

We shall also use the mixed curvature relation
\begin{align}
R^{M}(U_1,F_1,U_2,F_2)
={}&
g_1\left((\nabla_{U_1}T)_{F_1}U_2,F_2\right)
-
g_1\left((\nabla_{F_1}A)_{U_1}U_2,F_2\right)
\nonumber\\
&-
g_1(A_{U_1}F_1,A_{U_2}F_2)
+
g_1(T_{F_1}U_1,T_{F_2}U_2).
\label{eq-Gauss-mix-1}
\end{align}

\subsection{Ricci curvature associated with the distributions}

For a unit vertical vector $V_1\in\mathcal V_p$, we define the vertical
Ricci curvatures by
\begin{equation}
\operatorname{Ric}_{\mathcal V}^{M}(V_1)
=
\sum_{j=1}^{r}
R^{M}(V_1,V_j,V_j,V_1), \quad \operatorname{Ric}^{\mathcal V}(V_1)
=
\sum_{j=1}^{r}
R^{\mathcal V}(V_1,V_j,V_j,V_1).
\label{eq-Ric-vert}
\end{equation}

The corresponding Ricci curvatures with respect to the Levi--Civita
connection $\nabla^{0}$ are denoted by
\begin{equation}
\operatorname{Ric}^{0M}(V_1)
=
\sum_{j=2}^{r}
R^{0M}(V_1,V_j,V_j,V_1), \quad \operatorname{Ric}^{0\mathcal V}(V_1)
=
\sum_{j=2}^{r}
R^{0\mathcal V}(V_1,V_j,V_j,V_1).
\label{eq-Ric0}
\end{equation}

For a unit horizontal vector $h_1\in\mathcal H_p$, we define
\begin{equation}
\operatorname{Ric}_{\mathcal H}^{M}(h_1)
=
\sum_{j=2}^{s}
R^{M}(h_1,h_j,h_j,h_1), \quad \operatorname{Ric}^{\mathcal H}(h_1)
=
\sum_{j=2}^{s}
R^{\mathcal H}(h_1,h_j,h_j,h_1).
\label{eq-Ric-hor}
\end{equation}

The mixed Ricci contribution is defined by
\begin{equation}
\operatorname{Ric}^{\mathrm{mix}}(M)
=
\sum_{i=1}^{s}
\operatorname{Ric}_{\mathcal V}^{M}(h_i)
=
\sum_{i=1}^{s}\sum_{j=1}^{r}
R^{M}(h_i,V_j,V_j,h_i).
\label{eq-Ric-mix}
\end{equation}

\subsection{Scalar curvature and horizontal scalar curvature}

Let $\tau^{M}$ denote the scalar curvature of $(M,\nabla,g_1)$. With
respect to the adapted orthonormal frame, it decomposes as
\begin{equation}
\begin{aligned}
\tau^{M}
={}&
\sum_{1\leq i<j\leq r}
R^{M}(V_i,V_j,V_j,V_i)+
\sum_{1\leq i<j\leq s}
R^{M}(h_i,h_j,h_j,h_i)+
\sum_{i=1}^{s}\sum_{j=1}^{r}
R^{M}(h_i,V_j,V_j,h_i).
&
\end{aligned}
\label{eq-scalar-curvature}
\end{equation}

The scalar curvature associated with the horizontal distribution is
defined by
\begin{equation}
\tau^{\mathcal H}
=
\sum_{1\leq i<j\leq s}
R^{\mathcal H}(h_i,h_j,h_j,h_i).
\label{eq-horizontal-scalar}
\end{equation}

Using the horizontal curvature relation
\eqref{eq-Gauss-hor}, we obtain
\begin{equation}
\begin{aligned}
\sum_{1\leq i<j\leq s}
R^{M}(h_i,h_j,h_j,h_i)
={}&
\tau^{\mathcal H}
+\frac12\|\sigma\|^{2}
-\frac12g_1(A,A^{*})
-\|A\|^{2}.
\end{aligned}
\label{eq-horizontal-scalar-relation}
\end{equation}

This relation will be used in the derivation of the Chen--Ricci and
Hineva inequalities along the mixed distribution.

\subsection{An algebraic inequality of Hineva}

We recall the following algebraic inequality, which will be used to
derive the lower Ricci curvature estimates.

\begin{lemma}
\label{Hineva}
Let $B=(b_{ij})$ be a symmetric $k\times k$ matrix, where $k\geq2$.
Put
\[
b=\operatorname{trace}B,
\qquad
a=\|B\|^{2}.
\]
Then
\begin{equation}
b_{11}\sum_{i=1}^{k}b_{ii}
-
\sum_{i=1}^{k}b_{1i}^{2}
\geq
\frac{k-1}{k^{2}}
\left[
2b^{2}-ka
-(k-2)|b|
\sqrt{\frac{ka-b^{2}}{k-1}}
\right].
\label{eq-Hineva-algebraic}
\end{equation}
Equality holds if and only if $B$ is diagonal of the form
\[
B=\operatorname{diag}(b_1,b_2,\ldots,b_2),
\]
where
\begin{equation}
b_1
=
\frac{b}{k}
\mp
\frac{k-1}{k}
\sqrt{\frac{ka-b^{2}}{k-1}},
\label{eq-Hineva-b1}
\end{equation}
and
\begin{equation}
b_2
=
\frac{b}{k}
\pm
\frac{1}{k}
\sqrt{\frac{ka-b^{2}}{k-1}}.
\label{eq-Hineva-b2}
\end{equation}
\end{lemma}

In the present setting, Lemma~\ref{Hineva} will be applied to the
symmetric component matrices
\[
(T_{ij}^{\alpha})
\qquad\text{and}\qquad
(T_{ij}^{*\alpha}),
\]
for each $\alpha=1,\ldots,s$. In conjunction with the Cauchy--Schwarz
inequality, this gives the Hineva-type lower estimates for the vertical
and mixed Ricci curvatures.

\subsection{Useful norm identities}

For convenience, we record the following identities. From
\eqref{eq-T-components}, we have
\begin{equation}
\|T\|^{2}
=
\sum_{\alpha=1}^{s}
\sum_{i,j=1}^{r}
(T_{ij}^{\alpha})^{2},
\qquad
\|T^{*}\|^{2}
=
\sum_{\alpha=1}^{s}
\sum_{i,j=1}^{r}
(T_{ij}^{*\alpha})^{2},
\label{eq-T-norm-final}
\end{equation}
while
\begin{equation}
\|\operatorname{trace}T\|^{2}
=
\sum_{\alpha=1}^{s}
\left(
\sum_{i=1}^{r}T_{ii}^{\alpha}
\right)^{2},
\label{eq-traceT-norm}
\end{equation}
and
\begin{equation}
\|\operatorname{trace}T^{*}\|^{2}
=
\sum_{\alpha=1}^{s}
\left(
\sum_{i=1}^{r}T_{ii}^{*\alpha}
\right)^{2}.
\label{eq-traceTstar-norm}
\end{equation}

Similarly,
\begin{equation}
\|\sigma\|^{2}
=
\sum_{\alpha=1}^{r}
\left(
\sum_{i=1}^{s}A_{ii}^{\alpha}
\right)^{2},
\qquad
\|\sigma^{*}\|^{2}
=
\sum_{\alpha=1}^{r}
\left(
\sum_{i=1}^{s}A_{ii}^{*\alpha}
\right)^{2}.
\label{eq-sigma-norm-final}
\end{equation}

All the notation introduced above will be used in the subsequent sections
to derive the Chen--Ricci and Hineva inequalities along the vertical,
horizontal, and mixed distributions.

\subsection{Fix notations}
we fix these notations%
\begin{eqnarray*}
\zeta  &=&\frac{r-1}{2r^{2}}\left( 2\left\Vert {\rm trace}T\right\Vert
^{2}-r\left\Vert T\right\Vert ^{2}-\left( r-2\right) \left\Vert {\rm trace}%
T\right\Vert \sqrt{\frac{r\left\Vert T\right\Vert ^{2}-\left\Vert {\rm trace}%
T\right\Vert ^{2}}{r-1}}\right)  \\
\zeta ^{\ast } &=&\frac{r-1}{2r^{2}}\left( 2\left\Vert {\rm trace}T^{\ast
}\right\Vert ^{2}-r\left\Vert T^{\ast }\right\Vert ^{2}-\left( r-2\right)
\left\Vert {\rm trace}T^{\ast }\right\Vert \sqrt{\frac{r\left\Vert T^{\ast
}\right\Vert ^{2}-\left\Vert {\rm trace}T^{\ast }\right\Vert ^{2}}{r-1}}%
\right) 
\end{eqnarray*}%
\[
\varsigma =\frac{1}{2}\left\Vert \sigma \right\Vert ^{2}-\frac{1}{2}g\left( 
{\cal A},{\cal A}^{\ast }\right) -\left\Vert {\cal A}\right\Vert ^{2}-\delta
\left( N\right) -\delta ^{\ast }\left( N^{\ast }\right) -\delta \left(
\sigma \right) -\delta ^{\ast }\left( \sigma ^{\ast }\right) .
\]

\section{Chen-Ricci inequality for statistical submersion along vertical distribution}

In this section, we give an alternative proof of the Chen–Ricci inequality based on a simple and direct argument

\begin{theorem}\cite{SiddiquiChenSiddiqi2021}
\label{Theorem-CR-vert}Let $F:\left( M,\nabla ,g_{1}\right) \rightarrow
\left( N,\nabla ^{2},g_{2}\right) $ be a statistical submersion between two
statistical manifolds. Then for any unit vector $V_{1}\in {\cal V}_{p}$, we
have%
\begin{eqnarray*}
Ric_{{\cal V}}^{M}\left( V_{1}\right) &\leq &Ric^{{\cal V}}\left(
V_{1}\right) +2Ric^{0M}\left( V_{1}\right) -2Ric^{0{\cal V}}\left(
V_{1}\right) \\
&&+\frac{1}{8}\left( \left\Vert {\rm trace}T\right\Vert ^{2}+\left\Vert {\rm %
trace}T^{\ast }\right\Vert ^{2}\right)
\end{eqnarray*}%
The equality holds if and only if%
\[
{\cal T}_{11}^{\alpha }={\cal T}_{22}^{\alpha }+\cdots +{\cal T}%
_{rr}^{\alpha },\quad {\cal T}_{1j}^{\alpha }=0,\quad j\in \left\{ 2,\ldots
,r\right\} 
\]%
\[
{\cal T}_{11}^{\ast \alpha }={\cal T}_{22}^{\ast \alpha }+\cdots +{\cal T}%
_{rr}^{\ast \alpha },\quad {\cal T}_{1j}^{\ast \alpha }=0,\quad j\in \left\{
2,\ldots ,r\right\} 
\]
\end{theorem}

\noindent {\bf Proof.} From (\ref{eq-Gauss-stat}), we obtain%
\begin{equation}
\sum_{j=1}^{r}R^{M}\left( V_{1},V_{j},V_{j},V_{1}\right) =\sum_{j=1}^{r}R^{%
{\cal V}}\left( V_{1},V_{j},V_{j},V_{1}\right) +\sum_{j=1}^{r}g\left(
T_{V_{1}}V_{j},T_{V_{j}}^{\ast }V_{1}\right) -\sum_{j=1}^{r}g\left(
T_{V_{j}}V_{j},T_{V_{1}}^{\ast }V_{1}\right)  \label{eq-Ric-Tij}
\end{equation}%
Using (\ref{eq-T-components}) and (\ref{eq-Ric-vert}) in (\ref{eq-Ric-Tij}), we obtain%
\begin{equation}
Ric_{{\cal V}}^{M}\left( V_{1}\right) =Ric^{{\cal V}}\left( V_{1}\right)
+\sum_{\alpha =1}^{s}\sum_{j=1}^{r}T_{1j}^{\alpha }T_{1j}^{\ast \alpha
}-\sum_{\alpha =1}^{s}\sum_{j=1}^{r}T_{jj}^{\alpha }T_{11}^{\ast \alpha }
\label{eq-Ric-Tij-1}
\end{equation}%
Using (\ref{eq-T0-components}) in (\ref{eq-Ric-Tij-1}), we obtain%
\begin{eqnarray}
Ric_{{\cal V}}^{M}\left( V_{1}\right) &=&Ric^{{\cal V}}\left( V_{1}\right)
+2\sum_{\alpha =1}^{s}\sum_{j=2}^{r}\left( T_{1j}^{0\alpha }\right) ^{2}-%
\frac{1}{2}\sum_{\alpha =1}^{s}\sum_{j=2}^{r}\left( T_{1j}^{\alpha }\right)
^{2}-\frac{1}{2}\sum_{\alpha =1}^{s}\sum_{j=2}^{r}\left( T_{1j}^{\ast \alpha
}\right) ^{2}  \nonumber \\
&&-2\sum_{\alpha =1}^{s}\left( T_{11}^{0\alpha
}\sum_{j=2}^{r}T_{jj}^{0\alpha }\right) +\frac{1}{2}\sum_{\alpha
=1}^{s}\left( T_{11}^{\alpha }\sum_{j=2}^{r}T_{jj}^{\alpha }\right) +\frac{1%
}{2}\sum_{\alpha =1}^{s}\left( T_{11}^{\ast \alpha
}\sum_{j=2}^{r}T_{jj}^{\ast \alpha }\right)  \label{eq-abc}
\end{eqnarray}%
Thus, we obtain%
\begin{eqnarray}
Ric_{{\cal V}}^{M}\left( V_{1}\right) &\leq &Ric^{{\cal V}}\left(
V_{1}\right) +2\sum_{\alpha =1}^{s}\sum_{j=2}^{r}\left( T_{1j}^{0\alpha
}\right) ^{2}-2\sum_{\alpha =1}^{s}\left( T_{11}^{0\alpha
}\sum_{j=2}^{r}T_{jj}^{0\alpha }\right)  \nonumber \\
&&+\frac{1}{2}\sum_{\alpha =1}^{s}\left( T_{11}^{\alpha
}\sum_{j=2}^{r}T_{jj}^{\alpha }\right) +\frac{1}{2}\sum_{\alpha
=1}^{s}\left( T_{11}^{\ast \alpha }\sum_{j=2}^{r}T_{jj}^{\ast \alpha }\right)
\label{eq-CR-ineq}
\end{eqnarray}%
we now use the elementary inequality%
\[
d_{1}d_{2}\leq \frac{\left( d_{1}+d_{2}\right) ^{2}}{4} 
\]%
with equality if and only if $d_{1}=d_{2}$, Suppose $d_{1}=T_{11}^{\alpha }$
and $d_{2}=\sum_{j=2}^{r}T_{jj}^{\alpha }$, and similarly for the dual
components, then from (\ref{eq-CR-ineq}), we obtain%
\begin{eqnarray}
Ric_{{\cal V}}^{M}\left( V_{1}\right) &\leq &Ric^{{\cal V}}\left(
V_{1}\right) +2\sum_{\alpha =1}^{s}\sum_{j=2}^{r}\left( T_{1j}^{0\alpha
}\right) ^{2}-2\sum_{\alpha =1}^{s}\left( T_{11}^{0\alpha
}\sum_{j=2}^{r}T_{jj}^{0\alpha }\right)  \nonumber \\
&&+\frac{1}{8}\sum_{\alpha =1}^{s}\left( T_{11}^{\alpha
}+\sum_{j=2}^{r}T_{jj}^{\alpha }\right) ^{2}+\frac{1}{8}\sum_{\alpha
=1}^{s}\left( T_{11}^{\ast \alpha }+\sum_{j=2}^{r}T_{jj}^{\ast \alpha
}\right) ^{2}  \label{eq-CR-ineq-1}
\end{eqnarray}%
Using (\ref{eq-trace-T}) in (\ref{eq-CR-ineq-1}), we obtain%
\begin{eqnarray}
Ric_{{\cal V}}^{M}\left( V_{1}\right) &\leq &Ric^{{\cal V}}\left(
V_{1}\right) +2\sum_{\alpha =1}^{s}\sum_{j=2}^{r}\left( T_{1j}^{0\alpha
}\right) ^{2}-2\sum_{\alpha =1}^{s}\left( T_{11}^{0\alpha
}\sum_{j=2}^{r}T_{jj}^{0\alpha }\right)  \nonumber \\
&&+\frac{1}{8}\left( \left\Vert {\rm trace}T\right\Vert ^{2}+\left\Vert {\rm %
trace}T^{\ast }\right\Vert ^{2}\right)  \label{eq-CR-ineq-2}
\end{eqnarray}%
From (\ref{eq-Gauss-stat}), we obtain%
\begin{equation}
\sum_{j=2}^{r}R^{M}\left( V_{1},V_{j},V_{j},V_{1}\right) =\sum_{j=2}^{r}R^{%
{\cal V}}\left( V_{1},V_{j},V_{j},V_{1}\right) +\sum_{\alpha
=1}^{s}\sum_{j=2}^{r}\left( T_{1j}^{0\alpha }\right) ^{2}-\sum_{\alpha
=r+1}^{s}\sum_{j=2}^{r}T_{11}^{0\alpha }T_{jj}^{0\alpha }
\label{eq-CR-ineq-3}
\end{equation}%
Using (\ref{eq-Ric0}) in (\ref{eq-CR-ineq-3}), we obtain%
\begin{equation}
2Ric^{0M}\left( V_{1}\right) -2Ric^{0{\cal V}}\left( V_{1}\right)
=2\sum_{\alpha =1}^{s}\sum_{j=2}^{r}\left( T_{1j}^{0\alpha }\right)
^{2}-2\sum_{\alpha =1}^{s}\left( T_{11}^{0\alpha
}\sum_{j=2}^{r}T_{jj}^{0\alpha }\right)  \label{eq-CR-ineq-4}
\end{equation}%
From (\ref{eq-CR-ineq-2}) and (\ref{eq-CR-ineq-4}), we obtain%
\begin{eqnarray}
Ric_{{\cal V}}^{M}\left( V_{1}\right) &\leq &Ric^{{\cal V}}\left(
V_{1}\right) +2Ric^{0M}\left( V_{1}\right) -2Ric^{0{\cal V}}\left(
V_{1}\right)  \nonumber \\
&&+\frac{1}{8}\left( \left\Vert {\rm trace}T\right\Vert ^{2}+\left\Vert {\rm %
trace}T^{\ast }\right\Vert ^{2}\right)  \label{eq-CR-ineq-5}
\end{eqnarray}%
The equality holds in (\ref{eq-CR-ineq-5}) if and only if equality holds in (%
\ref{eq-CR-ineq-1}), for each $\alpha \in \left\{ 1,\ldots ,s\right\} $, we
obtain%
\begin{eqnarray}
{\cal T}_{11}^{\alpha } &=&{\cal T}_{22}^{\alpha }+\cdots +{\cal T}%
_{rr}^{\alpha },\quad \quad {\cal T}_{1j}^{\alpha }=0,\quad j\in \{2,\ldots
,r\}  \nonumber \\
{\cal T}_{11}^{\ast \alpha } &=&{\cal T}_{22}^{\ast \alpha }+\cdots +{\cal T}%
_{rr}^{\ast \alpha },\quad \quad {\cal T}_{1j}^{\ast \alpha }=0,\quad j\in
\{2,\ldots ,r\}  \label{eq-RM-6}
\end{eqnarray}%
Assuming that equality holds for all unit vector field $h_{i}$, $i\in
\{1,\ldots ,r\}$. Then from (\ref{eq-RM-6}), for each $\alpha \in \left\{
1,\ldots ,s\right\} $, we obtain%
\begin{equation}
{\cal T}_{ij}^{\alpha }=0,\quad {\cal T}_{ij}^{\ast \alpha }=0,\quad i\neq j,
\label{eq-RM-7}
\end{equation}%
\begin{eqnarray}
2{\cal T}_{ii}^{\alpha } &=&{\cal T}_{11}^{\alpha }+{\cal T}_{22}^{\alpha
}+\cdots +{\cal T}_{rr}^{\alpha },\quad i\in \{1,\ldots r\},  \nonumber \\
2{\cal T}_{ii}^{\ast \alpha } &=&{\cal T}_{11}^{\ast \alpha }+{\cal T}%
_{22}^{\ast \alpha }+\cdots +{\cal T}_{rr}^{\ast \alpha },\quad i\in
\{1,\ldots r\}.  \label{eq-RM-8}
\end{eqnarray}%
From (\ref{eq-RM-8}), we obtain%
\begin{eqnarray}
\left( r-2\right) ({\cal T}_{11}^{\alpha }+{\cal T}_{22}^{\alpha }+\cdots +%
{\cal T}_{rr}^{\alpha }) &=&0,  \nonumber \\
\left( r-2\right) ({\cal T}_{11}^{\ast \alpha }+{\cal T}_{22}^{\ast \alpha
}+\cdots +{\cal T}_{rr}^{\ast \alpha }) &=&0,  \label{eq-RM-9}
\end{eqnarray}%
From (\ref{eq-RM-9}), we obtain either $r=2$ or ${\cal T}_{11}^{\alpha }+%
{\cal T}_{22}^{\alpha }+\cdots +{\cal T}_{rr}^{\alpha }=0$ and ${\cal T}%
_{11}^{\ast \alpha }+{\cal T}_{22}^{\ast \alpha }+\cdots +{\cal T}%
_{rr}^{\ast \alpha }=0$. If ${\cal T}_{11}^{\alpha }+{\cal T}_{22}^{\alpha
}+\cdots +{\cal T}_{rr}^{\alpha }=0$ and ${\cal T}_{11}^{\ast \alpha }+{\cal %
T}_{22}^{\ast \alpha }+\cdots +{\cal T}_{rr}^{\ast \alpha }=0$, then from (%
\ref{eq-RM-8}), we obtain ${\cal T}_{ii}^{\alpha }=0$ and ${\cal T}%
_{ii}^{\ast \alpha }=0$ for all $\alpha \in \left\{ 1,\ldots ,s\right\} $
and $i\in \{1,\ldots ,r\}$. In view of (\ref{eq-RM-7}), we obtain ${\cal T}%
_{ij}^{\alpha }=0$ and ${\cal T}_{ij}^{\ast \alpha }=0$, for all $\alpha \in
\left\{ 1,\ldots ,s\right\} $ and $i,j\in \{1,\ldots ,r\}$. Which shows
fibre of $F$ is totally geodesic. If $r=2$, then from (\ref{eq-RM-7}) and (%
\ref{eq-RM-8}), for each $\alpha \in \left\{ 1,\ldots ,s\right\} $, we
obtain 
\begin{equation}
{\cal T}_{11}^{\alpha }={\cal T}_{22}^{\alpha },\quad {\cal T}_{12}^{\alpha
}=0,\quad {\cal T}_{11}^{\ast \alpha }={\cal T}_{22}^{\ast \alpha },\quad 
{\cal T}_{12}^{\ast \alpha }=0.  \label{eq-RM-10}
\end{equation}%
From (\ref{eq-RM-10}), it is straightforward that the fibre of $F$ is
totally umbilical.

\subsection{Hineva inequality for statistical submersion along vertical distribution}

In this section we obtain Hineva inequality

\begin{theorem}
\label{Theorem-Hineva-vert}Let $F:\left( M,\nabla ,g_{1}\right) \rightarrow
\left( N,\nabla ^{2},g_{2}\right) $ be a statistical submersion between two
statistical manifolds. Then for any unit vector $V_{1}\in {\cal V}_{p}$, we
have%
\begin{eqnarray*}
Ric_{{\cal V}}^{M}\left( V_{1}\right) &\geq &Ric^{{\cal V}}\left(
V_{1}\right) +2Ric^{0M}\left( V_{1}\right) -2Ric^{0{\cal V}}\left(
V_{1}\right) +\frac{r-1}{2r^{2}}\left( 2\left\Vert {\rm trace}T\right\Vert
^{2}-r\left\Vert T\right\Vert ^{2}\right. \\
&&\left. -\left( r-2\right) \left\Vert {\rm trace}T\right\Vert \sqrt{\frac{%
r\left\Vert T\right\Vert ^{2}-\left\Vert {\rm trace}T\right\Vert ^{2}}{r-1}}%
\right) +\frac{r-1}{2r^{2}}\left( 2\left\Vert {\rm trace}T^{\ast
}\right\Vert ^{2}-r\left\Vert T^{\ast }\right\Vert ^{2}\right. \\
&&\left. -\left( r-2\right) \left\Vert {\rm trace}T^{\ast }\right\Vert \sqrt{%
\frac{r\left\Vert T^{\ast }\right\Vert ^{2}-\left\Vert {\rm trace}T^{\ast
}\right\Vert ^{2}}{r-1}}\right) .
\end{eqnarray*}%
The equality holds if and only if for each $\alpha \in \{1,\ldots ,s\}$, the
matrix $(T_{ij}^{\alpha })$ and $(T_{ij}^{\ast \alpha })$ take the following
forms 
\[
T_{ij}^{\alpha }={\rm diag}(a_{\alpha },b_{\alpha },\ldots ,b_{\alpha
})\quad {\rm and\quad }T_{ij}^{\ast \alpha }={\rm diag}(a_{\alpha }^{\ast
},b_{\alpha }^{\ast },\ldots ,b_{\alpha }^{\ast }) 
\]%
and%
\[
\beta _{1}|a_{\alpha }-b_{\alpha }|+\beta _{2}\left( a_{\alpha
}+(r-1)b_{\alpha }\right) =0,\quad \alpha \in \{1,\ldots ,s\}. 
\]%
\[
\beta _{1}|a_{\alpha }^{\ast }-b_{\alpha }^{\ast }|+\beta _{2}\left(
a_{\alpha }^{\ast }+(r-1)b_{\alpha }^{\ast }\right) =0,\quad \alpha \in
\{1,\ldots ,s\}. 
\]%
where $\beta _{1}$ and $\beta _{2}$ are scalars but not both zero, $%
a_{\alpha }$, $b_{\alpha }$ and $a_{\alpha }^{\ast }$, $b_{\alpha }^{\ast }$
are two eigen values of each matrices $(B_{ij}^{\alpha })$ and $%
(B_{ij}^{\ast \alpha })$, respectively, such that%
\begin{eqnarray*}
a_{\alpha } &=&\frac{{\rm trace}\left( {\cal T}^{\alpha }\right) }{r}\mp 
\frac{r-1}{r}\sqrt{\frac{r{\rm trace}\left( {\cal T}^{\alpha }\right)
^{2}-\left( {\rm trace}{\cal T}^{\alpha }\right) ^{2}}{r-1}} \\
b_{\alpha } &=&\frac{{\rm trace}\left( {\cal T}^{\alpha }\right) }{r}\pm 
\frac{1}{r}\sqrt{\frac{r{\rm trace}\left( {\cal T}^{\alpha }\right)
^{2}-\left( {\rm trace}{\cal T}^{\alpha }\right) ^{2}}{r-1}}
\end{eqnarray*}%
and%
\begin{eqnarray*}
a_{\alpha }^{\ast } &=&\frac{{\rm trace}\left( {\cal T}^{\ast \alpha
}\right) }{r}\mp \frac{r-1}{r}\sqrt{\frac{r{\rm trace}\left( {\cal T}^{\ast
\alpha }\right) ^{2}-\left( {\rm trace}{\cal T}^{\ast \alpha }\right) ^{2}}{%
r-1}} \\
b_{\alpha }^{\ast } &=&\frac{{\rm trace}\left( {\cal T}^{\ast \alpha
}\right) }{r}\pm \frac{1}{r}\sqrt{\frac{r{\rm trace}\left( {\cal T}^{\ast
\alpha }\right) ^{2}-\left( {\rm trace}{\cal T}^{\ast \alpha }\right) ^{2}}{%
r-1}}
\end{eqnarray*}
\end{theorem}

\noindent {\bf Proof.} From (\ref{eq-abc}), we have%
\begin{eqnarray}
Ric_{{\cal V}}^{M}\left( V_{1}\right) &=&Ric^{{\cal V}}\left( V_{1}\right)
+2\sum_{\alpha =1}^{s}\sum_{j=2}^{r}\left( T_{1j}^{0\alpha }\right)
^{2}-2\sum_{\alpha =1}^{s}\left( T_{11}^{0\alpha
}\sum_{j=2}^{r}T_{jj}^{0\alpha }\right)  \nonumber \\
&&+\frac{1}{2}\sum_{\alpha =1}^{s}\left( T_{11}^{\alpha
}\sum_{j=2}^{r}T_{jj}^{\alpha }\right) -\frac{1}{2}\sum_{\alpha
=1}^{s}\sum_{j=2}^{r}\left( T_{1j}^{\alpha }\right) ^{2}  \nonumber \\
&&+\frac{1}{2}\sum_{\alpha =1}^{s}\left( T_{11}^{\ast \alpha
}\sum_{j=2}^{r}T_{jj}^{\ast \alpha }\right) -\frac{1}{2}\sum_{\alpha
=1}^{s}\sum_{j=2}^{r}\left( T_{1j}^{\ast \alpha }\right) ^{2}.
\label{eq-stat-hin-1}
\end{eqnarray}%
Using Lemma \ref{Hineva}, we obtain%
\begin{eqnarray}
T_{11}^{\alpha }\sum_{j=1}^{r}T_{jj}^{\alpha }-\sum_{j=1}^{r}\left(
T_{1j}^{\alpha }\right) ^{2} &\geq &\frac{r-1}{r^{2}}\left( 2\left(
\sum_{j=1}^{r}T_{jj}^{\alpha }\right) ^{2}-r\sum_{i,j=1}^{r}\left(
T_{ij}^{\alpha }\right) ^{2}-\left( r-2\right) \right.  \nonumber \\
&&\left. \left\vert \sum_{j=1}^{r}T_{jj}^{\alpha }\right\vert \sqrt{\frac{%
r\sum_{i,j=1}^{r}\left( T_{ij}^{\alpha }\right) ^{2}-\left(
\sum_{j=1}^{r}T_{jj}^{\alpha }\right) ^{2}}{r-1}}\right)
\label{eq-Lemma-Hineva-RM}
\end{eqnarray}%
and%
\begin{eqnarray}
T_{11}^{\ast \alpha }\sum_{j=1}^{r}T_{jj}^{\ast \alpha
}-\sum_{j=1}^{r}\left( T_{1j}^{\ast \alpha }\right) ^{2} &\geq &\frac{r-1}{%
r^{2}}\left( 2\left( \sum_{j=1}^{r}T_{jj}^{\ast \alpha }\right)
^{2}-r\sum_{i,j=1}^{r}\left( T_{ij}^{\ast \alpha }\right) ^{2}-\left(
r-2\right) \right.  \nonumber \\
&&\left. \left\vert \sum_{j=1}^{r}T_{jj}^{\ast \alpha }\right\vert \sqrt{%
\frac{r\sum_{i,j=1}^{r}\left( T_{ij}^{\ast \alpha }\right) ^{2}-\left(
\sum_{j=1}^{r}T_{jj}^{\ast \alpha }\right) ^{2}}{r-1}}\right)
\label{eq-Lemma-Hineva-RM-1}
\end{eqnarray}%
Using (\ref{eq-Lemma-Hineva-RM}) and (\ref{eq-Lemma-Hineva-RM-1}) in (\ref%
{eq-stat-hin-1}), we obtain%
\begin{eqnarray}
Ric_{{\cal V}}^{M}\left( V_{1}\right) &\geq &Ric^{{\cal V}}\left(
V_{1}\right) +2\sum_{\alpha =1}^{s}\sum_{j=2}^{r}\left( T_{1j}^{0\alpha
}\right) ^{2}-2\sum_{\alpha =1}^{s}\left( {\cal T}_{11}^{0\alpha
}\sum_{j=2}^{r}T_{jj}^{0\alpha }\right) +\frac{r-1}{2r^{2}}\left\{
2\sum_{\alpha =1}^{s}\left( \sum_{j=1}^{r}T_{jj}^{\alpha }\right) ^{2}\right.
\nonumber \\
&&\left. -r\sum_{\alpha =1}^{s}\sum_{i,j=1}^{r}\left( T_{ij}^{\alpha
}\right) ^{2}-\left( r-2\right) \sum_{\alpha =1}^{s}\left\vert
\sum_{j=1}^{r}T_{jj}^{\alpha }\right\vert \sqrt{\frac{r\sum_{i,j=1}^{r}%
\left( T_{ij}^{\alpha }\right) ^{2}-\left( \sum_{j=1}^{r}T_{jj}^{\alpha
}\right) ^{2}}{r-1}}\right\}  \nonumber \\
&&+\frac{r-1}{2r^{2}}\left\{ 2\sum_{\alpha =1}^{s}\left(
\sum_{j=1}^{r}T_{jj}^{\ast \alpha }\right) ^{2}-r\sum_{\alpha
=1}^{s}\sum_{i,j=1}^{r}\left( T_{ij}^{\ast \alpha }\right) ^{2}-\left(
r-2\right) \right.  \nonumber \\
&&\left. \sum_{\alpha =1}^{s}\left\vert \sum_{j=1}^{r}T_{jj}^{\ast \alpha
}\right\vert \sqrt{\frac{r\sum_{i,j=1}^{r}\left( T_{ij}^{\ast \alpha
}\right) ^{2}-\left( \sum_{j=1}^{r}T_{jj}^{\ast \alpha }\right) ^{2}}{r-1}}%
\right\}  \label{eq-Ric-Hineva-RM-2}
\end{eqnarray}%
Using Cauchy-Schwartz inequality, we have%
\begin{equation}
\sum_{\alpha =1}^{s}\left\vert \sum_{j=1}^{r}T_{jj}^{\alpha }\right\vert 
\sqrt{\frac{r\sum_{i,j=1}^{r}\left( T_{ij}^{\alpha }\right) ^{2}-\left(
\sum_{j=1}^{r}T_{jj}^{\alpha }\right) ^{2}}{r-1}}\leq \left\Vert {\rm trace}%
T\right\Vert \sqrt{\frac{r\left\Vert T\right\Vert ^{2}-\left\Vert {\rm trace}%
T\right\Vert ^{2}}{r-1}},  \label{eq-cauch-squartz-Hineva-RM}
\end{equation}%
and%
\begin{equation}
\sum_{\alpha =1}^{s}\left\vert \sum_{j=1}^{r}T_{jj}^{\ast \alpha
}\right\vert \sqrt{\frac{r\sum_{i,j=1}^{r}\left( T_{ij}^{\ast \alpha
}\right) ^{2}-\left( \sum_{j=1}^{r}T_{jj}^{\ast \alpha }\right) ^{2}}{r-1}}%
\leq \left\Vert {\rm trace}T^{\ast }\right\Vert \sqrt{\frac{r\left\Vert
T^{\ast }\right\Vert ^{2}-\left\Vert {\rm trace}T^{\ast }\right\Vert ^{2}}{%
r-1}}.  \label{eq-cauch-squartz-Hineva-RM-1}
\end{equation}%
Using (\ref{eq-cauch-squartz-Hineva-RM}) and (\ref%
{eq-cauch-squartz-Hineva-RM-1}) in (\ref{eq-Ric-Hineva-RM-2}), we obtain%
\begin{eqnarray}
Ric_{{\cal V}}^{M}\left( V_{1}\right) &\geq &Ric^{{\cal V}}\left(
V_{1}\right) +2\sum_{\alpha =1}^{s}\sum_{j=2}^{r}\left( T_{1j}^{0\alpha
}\right) ^{2}-2\sum_{\alpha =1}^{s}\left( T_{11}^{0\alpha
}\sum_{j=2}^{r}T_{jj}^{0\alpha }\right) +\frac{r-1}{2r^{2}}\left(
2\left\Vert {\rm trace}T\right\Vert ^{2}\right.  \nonumber \\
&&\left. -r\left\Vert T\right\Vert ^{2}-\left( r-2\right) \left\Vert {\rm %
trace}T\right\Vert \sqrt{\frac{r\left\Vert T\right\Vert ^{2}-\left\Vert {\rm %
trace}T\right\Vert ^{2}}{r-1}}\right) +\frac{r-1}{2r^{2}}\left( 2\left\Vert 
{\rm trace}T^{\ast }\right\Vert ^{2}\right.  \nonumber \\
&&\left. -r\left\Vert T^{\ast }\right\Vert ^{2}-\left( r-2\right) \left\Vert 
{\rm trace}T^{\ast }\right\Vert \sqrt{\frac{r\left\Vert T^{\ast }\right\Vert
^{2}-\left\Vert {\rm trace}T^{\ast }\right\Vert ^{2}}{r-1}}\right) .
\label{eq-Ric-Hineva-RM-3}
\end{eqnarray}%
From (\ref{eq-CR-ineq-4}), we have%
\begin{equation}
2Ric^{0M}\left( V_{1}\right) -2Ric^{0{\cal V}}\left( V_{1}\right)
=2\sum_{\alpha =1}^{s}\sum_{j=2}^{r}\left( T_{1j}^{0\alpha }\right)
^{2}-2\sum_{\alpha =1}^{s}\left( T_{11}^{0\alpha
}\sum_{j=2}^{r}T_{jj}^{0\alpha }\right)  \label{eq-Ric-Hineva-RM-4}
\end{equation}%
Using (\ref{eq-Ric-Hineva-RM-4}) in (\ref{eq-Ric-Hineva-RM-3}), we obtain%
\begin{eqnarray}
Ric_{{\cal V}}^{M}\left( V_{1}\right) &\geq &Ric^{{\cal V}}\left(
V_{1}\right) +2Ric^{0M}\left( V_{1}\right) -2Ric^{0{\cal V}}\left(
V_{1}\right) +\frac{r-1}{2r^{2}}\left( 2\left\Vert {\rm trace}T\right\Vert
^{2}-r\left\Vert T\right\Vert ^{2}\right.  \nonumber \\
&&\left. -\left( r-2\right) \left\Vert {\rm trace}T\right\Vert \sqrt{\frac{%
r\left\Vert T\right\Vert ^{2}-\left\Vert {\rm trace}T\right\Vert ^{2}}{r-1}}%
\right) +\frac{r-1}{2r^{2}}\left( 2\left\Vert {\rm trace}T^{\ast
}\right\Vert ^{2}-r\left\Vert T^{\ast }\right\Vert ^{2}\right.  \nonumber \\
&&\left. -\left( r-2\right) \left\Vert {\rm trace}T^{\ast }\right\Vert \sqrt{%
\frac{r\left\Vert T^{\ast }\right\Vert ^{2}-\left\Vert {\rm trace}T^{\ast
}\right\Vert ^{2}}{r-1}}\right) .  \label{eq-Ric-Hineva-RM-5}
\end{eqnarray}%
Now, we establish the equality case. Equality in the desired inequality is
attained if and only if equality holds simultaneously in (\ref%
{eq-Lemma-Hineva-RM}), (\ref{eq-Lemma-Hineva-RM-1}), (\ref%
{eq-cauch-squartz-Hineva-RM}) and (\ref{eq-cauch-squartz-Hineva-RM-1}). If
equality holds in (\ref{eq-Lemma-Hineva-RM}) and (\ref{eq-Lemma-Hineva-RM-1}%
), then, for each $\alpha \in \{1,\ldots ,s\}$, the matrix $(T_{ij}^{\alpha
})$ and $(T_{ij}^{\ast \alpha })$ have the forms 
\[
T_{ij}^{\alpha }={\rm diag}(a_{\alpha },b_{\alpha },\ldots ,b_{\alpha
})\quad {\rm and\quad }T_{ij}^{\ast \alpha }={\rm diag}(a_{\alpha }^{\ast
},b_{\alpha }^{\ast },\ldots ,b_{\alpha }^{\ast }) 
\]%
On the other hand, equality in (\ref{eq-cauch-squartz-Hineva-RM}) and (\ref%
{eq-cauch-squartz-Hineva-RM-1}) hold if and only if the vectors involved in
the Cauchy--Schwarz inequality are linearly dependent. Consequently, there
exist constants $\beta _{1},\beta _{2}$, not both zero, such that 
\[
\beta _{1}\left\vert \sum_{j=1}^{r}T_{jj}^{\alpha }\right\vert +\beta _{2}%
\sqrt{\frac{r\sum_{i,j=1}^{r}(T_{ij}^{\alpha })^{2}-\left(
\sum_{j=1}^{r}T_{jj}^{\alpha }\right) ^{2}}{r-1}}=0,\quad \alpha \in
\{1,\ldots ,s\}. 
\]%
and%
\[
\beta _{1}\left\vert \sum_{j=1}^{r}T_{jj}^{\ast \alpha }\right\vert +\beta
_{2}\sqrt{\frac{r\sum_{i,j=1}^{r}(T_{ij}^{\ast \alpha })^{2}-\left(
\sum_{j=1}^{r}T_{jj}^{\ast \alpha }\right) ^{2}}{r-1}}=0,\quad \alpha \in
\{1,\ldots ,s\}. 
\]%
which implies 
\[
\beta _{1}|a_{\alpha }-b_{\alpha }|+\beta _{2}\left( a_{\alpha
}+(r-1)b_{\alpha }\right) =0,\quad \alpha \in \{1,\ldots ,s\}. 
\]%
and%
\[
\beta _{1}|a_{\alpha }^{\ast }-b_{\alpha }^{\ast }|+\beta _{2}\left(
a_{\alpha }^{\ast }+(r-1)b_{\alpha }^{\ast }\right) =0,\quad \alpha \in
\{1,\ldots ,s\}. 
\]%
This completes the proof.

\subsection{Simultaneous Chen--Ricci and Hineva inequalities for statistical
submersion along vertical distyribution}

In this subsection, we combine Theorems~\ref{Theorem-CR-vert} and~\ref%
{Theorem-Hineva-vert} to characterize the simultaneous equality cases of the
obtained Ricci curvature inequalities.

\begin{theorem}
\label{Th-ineq-both-side-vert}Let $F:(N_{1},g_{1})\rightarrow (N_{2},g_{2})$
be a Riemannian submersion between two Riemannian manifolds. Then%
\begin{eqnarray}
&&Ric_{{\cal V}}^{M}\left( V_{1}\right) -2Ric^{0M}\left( V_{1}\right)
+2Ric^{0{\cal V}}\left( V_{1}\right) -+\frac{1}{8}\left( \left\Vert {\rm %
trace}T\right\Vert ^{2}+\left\Vert {\rm trace}T^{\ast }\right\Vert
^{2}\right)  \nonumber \\
&\leq &Ric^{{\cal V}}\left( V_{1}\right) \leq Ric_{{\cal V}}^{M}\left(
V_{1}\right) -2Ric^{0M}\left( V_{1}\right) +2Ric^{0{\cal V}}\left(
V_{1}\right) -\zeta -\zeta ^{\ast }  \label{ineq-bot-side}
\end{eqnarray}%
The equalities in both inequalities of {\rm (\ref{ineq-bot-side})} hold
simultaneously for a unit vertical vector $h$ if and only if the matrices 
\[
{\cal T}^{\alpha }={\rm diag}\left( b_{\alpha },\ldots ,\underset{i^{th}\ 
{\rm position}}{\underbrace{a_{\alpha }}},\ldots ,b_{\alpha }\right) ,\quad
V=V_{i},\quad i=1,\ldots ,r,\quad \alpha \in \left\{ 1,\ldots ,s\right\} , 
\]%
\[
{\cal T}^{\ast \alpha }={\rm diag}\left( b_{\alpha }^{\ast },\ldots ,%
\underset{i^{th}\ {\rm position}}{\underbrace{a_{\alpha }^{\ast }}},\ldots
,b_{\alpha }^{\ast }\right) ,\quad h=h_{i},\quad i=1,\ldots ,r,\quad \alpha
\in \left\{ 1,\ldots ,s\right\} , 
\]%
and 
\[
\beta _{1}|a_{\alpha }-b_{\alpha }|+\beta _{2}\left( a_{\alpha
}+(r-1)b_{\alpha }\right) =0,\qquad \alpha \in \left\{ 1,\ldots ,s\right\} , 
\]%
\[
\beta _{1}|a_{\alpha }^{\ast }-b_{\alpha }^{\ast }|+\beta _{2}\left(
a_{\alpha }^{\ast }+(r-1)b_{\alpha }^{\ast }\right) =0,\qquad \alpha \in
\left\{ 1,\ldots ,s\right\} , 
\]%
where $\beta _{1}$ and $\beta _{2}$ are scalars but not both zero. Moreover, 
\[
a_{\alpha }=\frac{{\rm trace}{\cal T}^{\alpha }}{2},\qquad b_{\alpha }=\frac{%
{\rm trace}{\cal T}^{\alpha }}{2(r-1)}. 
\]%
\[
a_{\alpha }^{\ast }=\frac{{\rm trace}{\cal T}^{\ast \alpha }}{2},\qquad
b_{\alpha }^{\ast }=\frac{{\rm trace}{\cal T}^{\ast \alpha }}{2(r-1)}. 
\]%
Furthermore, the equalities in both inequalities of {\rm (\ref{ineq-bot-side}%
)} hold simultaneously for every unit vertical vector $h_{i}$, $i=1,\ldots
,r $, if and only if either the map $F$ is totally geodesic or $r=2$ and 
\[
{\cal T}_{11}^{\alpha }={\cal T}_{22}^{\alpha },\qquad {\cal T}_{12}^{\alpha
}=0,\qquad \alpha =1,\ldots ,s, 
\]%
\[
{\cal T}_{11}^{\ast \alpha }={\cal T}_{22}^{\ast \alpha },\qquad {\cal T}%
_{12}^{\ast \alpha }=0,\qquad \alpha =1,\ldots ,s, 
\]%
it is straightforward that the map $F$ is totally umbilical.
\end{theorem}

\noindent {\bf Proof.} Assume that the equalities in both inequalities of (%
\ref{ineq-bot-side}) hold simultaneously for the unit vertical vector $%
V=V_{1}$. From the equality case of Theorem~\ref{Theorem-CR-vert}, we obtain 
\begin{eqnarray}
{\cal T}_{1j}^{\alpha } &=&0,\qquad j=2,\ldots ,r,\quad \alpha =1,\ldots ,s,
\nonumber \\
{\cal T}_{1j}^{\ast \alpha } &=&0,\qquad j=2,\ldots ,r,\quad \alpha
=1,\ldots ,s  \label{eq-eq-1}
\end{eqnarray}%
and 
\begin{eqnarray}
{\cal T}_{11}^{\alpha } &=&{\cal T}_{22}^{\alpha }+\cdots +{\cal T}%
_{rr}^{\alpha },\qquad \alpha =1,\ldots ,s.  \nonumber \\
{\cal T}_{11}^{\ast \alpha } &=&{\cal T}_{22}^{\ast \alpha }+\cdots +{\cal T}%
_{rr}^{\ast \alpha },\qquad \alpha =1,\ldots ,s.  \label{eq-eq-2}
\end{eqnarray}%
On the other hand, the equality case of Theorem~\ref{Theorem-Hineva-vert}
yields 
\begin{eqnarray}
{\cal T}_{11}^{\alpha } &=&a_{\alpha },\qquad {\cal T}_{22}^{\alpha }=\cdots
={\cal T}_{rr}^{\alpha }=b_{\alpha },\qquad \alpha =1,\ldots ,s.  \nonumber
\\
{\cal T}_{11}^{\ast \alpha } &=&a_{\alpha }^{\ast },\qquad {\cal T}%
_{22}^{\ast \alpha }=\cdots ={\cal T}_{rr}^{\ast \alpha }=b_{\alpha }^{\ast
},\qquad \alpha =1,\ldots ,s.  \label{eq-eq-3}
\end{eqnarray}%
Combining (\ref{eq-eq-2}) and (\ref{eq-eq-3}), we obtain%
\[
2{\cal T}_{11}^{\alpha }={\rm trace}({\cal T}^{\alpha }),\quad 2{\cal T}%
_{11}^{\ast \alpha }={\rm trace}({\cal T}^{\ast \alpha }), 
\]%
which immediately gives 
\[
a_{\alpha }=\frac{{\rm trace}({\cal T}^{\alpha })}{2},\quad a_{\alpha
}^{\ast }=\frac{{\rm trace}({\cal T}^{\ast \alpha })}{2}, 
\]%
From (\ref{eq-eq-2}) and (\ref{eq-eq-3}), we obtain 
\[
a_{\alpha }=(r-1)b_{\alpha },\quad a_{\alpha }^{\ast }=(r-1)b_{\alpha
}^{\ast }, 
\]%
consequently, 
\[
b_{\alpha }=\frac{{\rm trace}({\cal T}^{\alpha })}{2(r-1)},\quad b_{\alpha
}^{\ast }=\frac{{\rm trace}({\cal T}^{\ast \alpha })}{2(r-1)}. 
\]%
Suppose that the equalities in (\ref{ineq-bot-side}) hold simultaneously for
every unit vertical vector $V_{i}$, $i=1,\ldots ,r$. Then, for each $\alpha
\in \left\{ 1,\ldots ,s\right\} $, by the equality case of Theorem~\ref%
{Theorem-CR-vert}, 
\begin{equation}
{\cal T}_{ij}^{\alpha }=0,\quad {\cal T}_{ij}^{\ast \alpha }=0,\qquad i\neq
j,\quad i,j\in \{1,\ldots ,r\},  \label{eq-sim-mix-1}
\end{equation}%
and 
\begin{equation}
{\cal T}_{ii}^{\alpha }=\sum_{k=1,k\neq i}^{r}{\cal T}_{kk}^{\alpha },\qquad 
{\cal T}_{ii}^{\ast \alpha }=\sum_{k=1,k\neq i}^{r}{\cal T}_{kk}^{\ast
\alpha },\quad i=1,\ldots ,r.  \label{eq-sim-mix-2}
\end{equation}%
Furthermore, for each $i$, the equality case of Theorem~\ref%
{Theorem-Hineva-vert} implies 
\begin{equation}
{\cal T}_{ii}^{\alpha }=a_{\alpha },\ {\cal T}_{ii}^{\ast \alpha }=a_{\alpha
}^{\ast },\ {\cal T}_{kk}^{\alpha }=b_{\alpha },\ {\cal T}_{kk}^{\ast \alpha
}=b_{\alpha }^{\ast },\quad k\neq i,\quad k\in \left\{ 1,\ldots
,i-1,i+1,\ldots ,r\right\} .  \label{eq-sim-mix-3}
\end{equation}%
From (\ref{eq-sim-mix-2}) and (\ref{eq-sim-mix-3}), we obtain 
\[
2{\cal T}_{ii}^{\alpha }=a_{\alpha }+(r-1)b_{\alpha },\ 2{\cal T}_{ii}^{\ast
\alpha }=a_{\alpha }^{\ast }+(r-1)b_{\alpha }^{\ast },\quad i\in \left\{
1,\ldots ,r\right\} , 
\]%
which leads to 
\[
2\left( a_{\alpha }+(r-1)b_{\alpha }\right) =r\left( a_{\alpha
}+(r-1)b_{\alpha }\right) ,\quad 2\left( a_{\alpha }^{\ast }+(r-1)b_{\alpha
}^{\ast }\right) =r\left( a_{\alpha }^{\ast }+(r-1)b_{\alpha }^{\ast
}\right) . 
\]%
Therefore, 
\[
(r-2)\left( a_{\alpha }+(r-1)b_{\alpha }\right) =0,\quad (r-2)\left(
a_{\alpha }^{\ast }+(r-1)b_{\alpha }^{\ast }\right) =0. 
\]%
Thus, 
\[
{\rm either\ }r=2\ {\rm or}\ a_{\alpha }+(r-1)b_{\alpha }=0\ {\rm and}\
a_{\alpha }^{\ast }+(r-1)b_{\alpha }^{\ast }=0. 
\]%
If 
\[
a_{\alpha }+(r-1)b_{\alpha }=0\ {\rm and}\ a_{\alpha }^{\ast
}+(r-1)b_{\alpha }^{\ast }=0 
\]%
then 
\[
{\cal T}_{ii}^{\alpha }=0,\ {\cal T}_{ii}^{\ast \alpha }=0,\qquad {\cal T}%
_{ij}^{\alpha }=0,\ {\cal T}_{ij}^{\ast \alpha }=0,\quad (i\neq j), 
\]%
which implies that ${\cal T}_{ij}^{\alpha }=0$ and ${\cal T}_{ij}^{\ast
\alpha }=0$ for all $i,j\in \left\{ 1,\ldots ,r\right\} $ and $\alpha \in
\left\{ 1,\ldots ,s\right\} $. Hence, the fibres of $F$ are totally
geodesic. If $r=2$, then 
\[
{\cal T}_{11}^{\alpha }={\cal T}_{22}^{\alpha },\ {\cal T}_{11}^{\ast \alpha
}={\cal T}_{22}^{\ast \alpha },\qquad {\cal T}_{12}^{\alpha }=0,\ {\cal T}%
_{12}^{\ast \alpha }=0, 
\]%
it is straightforward that the fibres of $F$ are totally umbilical.

\begin{example}
Consider
\[
M=\mathbb{R}^{4},\qquad N=\mathbb{R}^{2},
\]
endowed with the standard Euclidean metrics
\[
g_{1}=dx^{2}+dy^{2}+dz^{2}+dw^{2},
\qquad
g_{2}=dx^{2}+dy^{2},
\]
and define the map
\[
F:M\longrightarrow N,\qquad
F(x,y,z,w)=(x,y).
\]
Let
\[
\mathcal{V}=\ker F_{*}
=\operatorname{span}\{V_{1},V_{2}\},
\qquad
\mathcal{H}=\mathcal{V}^{\perp}
=\operatorname{span}\{h_{1},h_{2}\},
\]
where
\[
V_{1}=\frac{\partial}{\partial z},\qquad
V_{2}=\frac{\partial}{\partial w},\qquad
h_{1}=\frac{\partial}{\partial x},\qquad
h_{2}=\frac{\partial}{\partial y}.
\]
Thus
\[
r=\dim\mathcal{V}=2,\qquad
s=\dim\mathcal{H}=2.
\]

Let $\nabla^{0}$ denote the Euclidean Levi--Civita connection on
$M$. Fix a nonzero constant $a$ and define a symmetric $(1,2)$-tensor
$K$ by
\[
K_{V_{1}}V_{1}=ah_{1},\qquad
K_{V_{2}}V_{2}=ah_{1},
\]
\[
K_{V_{1}}h_{1}=aV_{1},\qquad
K_{V_{2}}h_{1}=aV_{2},
\]
and let all remaining components of $K$ be zero. Define two affine
connections on $M$ by
\[
\nabla_{X}Y=\nabla^{0}_{X}Y+K_{X}Y,
\qquad
\nabla^{*}_{X}Y=\nabla^{0}_{X}Y-K_{X}Y.
\]
Since $K$ is symmetric and
\[
g_{1}(K_{X}Y,Z)=g_{1}(Y,K_{X}Z),
\]
the connections $\nabla$ and $\nabla^{*}$ are torsion-free and
conjugate with respect to $g_{1}$. Hence
\[
(M,\nabla,g_{1})
\]
is a statistical manifold.

For horizontal basic vector fields $X,Y$, we have
\[
F_{*}(\nabla_{X}Y)
=\nabla^{2}_{F_{*}X}F_{*}Y,
\]
where $\nabla^{2}$ is the Euclidean connection on $N$. Consequently,
\[
F:(M,\nabla,g_{1})\longrightarrow
(N,\nabla^{2},g_{2})
\]
is a statistical submersion.

We now compute the fundamental tensors appearing in the
Chen--Ricci and Hineva inequalities. Since
\[
T_{V_{i}}V_{j}
=\mathcal{H}\nabla_{V_{i}}V_{j},
\]
we obtain
\[
T_{V_{1}}V_{1}=ah_{1},\qquad
T_{V_{2}}V_{2}=ah_{1},
\]
and
\[
T_{V_{1}}V_{2}
=T_{V_{2}}V_{1}=0.
\]
Therefore,
\[
\left(T^{1}_{ij}\right)
=
\begin{pmatrix}
a&0\\
0&a
\end{pmatrix},
\qquad
\left(T^{2}_{ij}\right)
=
\begin{pmatrix}
0&0\\
0&0
\end{pmatrix}.
\]

Similarly, since
\[
T^{*}_{V_{i}}V_{j}
=\mathcal{H}\nabla^{*}_{V_{i}}V_{j},
\]
we have
\[
T^{*}_{V_{1}}V_{1}=-ah_{1},
\qquad
T^{*}_{V_{2}}V_{2}=-ah_{1},
\]
and
\[
T^{*}_{V_{1}}V_{2}
=T^{*}_{V_{2}}V_{1}=0.
\]
Hence
\[
\left(T^{*1}_{ij}\right)
=
\begin{pmatrix}
-a&0\\
0&-a
\end{pmatrix},
\qquad
\left(T^{*2}_{ij}\right)
=
\begin{pmatrix}
0&0\\
0&0
\end{pmatrix}.
\]

Consequently,
\[
T^{0\alpha}_{ij}
=\frac{1}{2}
\left(T^{\alpha}_{ij}+T^{*\alpha}_{ij}\right)
\]
gives
\[
T^{01}_{ij}=T^{02}_{ij}=0.
\]
Thus
\[
T^{0}=0.
\]

Furthermore,
\[
\operatorname{trace}T
=\sum_{i=1}^{2}T_{V_{i}}V_{i}
=2ah_{1},
\]
and therefore
\[
\|\operatorname{trace}T\|^{2}=4a^{2}.
\]
Similarly,
\[
\operatorname{trace}T^{*}=-2ah_{1},
\qquad
\|\operatorname{trace}T^{*}\|^{2}=4a^{2}.
\]
Moreover,
\[
\|T\|^{2}
=\sum_{\alpha=1}^{2}\sum_{i,j=1}^{2}
\left(T^{\alpha}_{ij}\right)^{2}
=2a^{2},
\]
and
\[
\|T^{*}\|^{2}=2a^{2}.
\]

Since the fibres of $F$ are Euclidean planes, their intrinsic
curvature vanishes. Hence
\[
\operatorname{Ric}^{\mathcal V}(V_{1})=0.
\]
Also, with respect to the Levi--Civita connection $\nabla^{0}$,
\[
\operatorname{Ric}^{0M}(V_{1})=0,
\qquad
\operatorname{Ric}^{0\mathcal V}(V_{1})=0.
\]

On the other hand, using the connection $\nabla$, we obtain
\[
R(V_{1},V_{2})V_{2}=a^{2}V_{1},
\]
and consequently
\[
\operatorname{Ric}^{M}_{\mathcal V}(V_{1})
=
R(V_{1},V_{2},V_{2},V_{1})
=a^{2}.
\]

We now verify the Chen--Ricci inequality. By Theorem~3.1,
\[
\operatorname{Ric}^{M}_{\mathcal V}(V_{1})
\leq
\operatorname{Ric}^{\mathcal V}(V_{1})
+2\operatorname{Ric}^{0M}(V_{1})
-2\operatorname{Ric}^{0\mathcal V}(V_{1})
+\frac{1}{8}
\left(
\|\operatorname{trace}T\|^{2}
+
\|\operatorname{trace}T^{*}\|^{2}
\right).
\]
Substituting the above quantities, we get
\[
\begin{aligned}
\operatorname{Ric}^{M}_{\mathcal V}(V_{1})
&\leq
0+2(0)-2(0)
+\frac18(4a^{2}+4a^{2})\\
&=a^{2}.
\end{aligned}
\]
Since
\[
\operatorname{Ric}^{M}_{\mathcal V}(V_{1})=a^{2},
\]
we conclude that
\[
\boxed{
\operatorname{Ric}^{M}_{\mathcal V}(V_{1})
=
a^{2}
=
\operatorname{Ric}^{\mathcal V}(V_{1})
+2\operatorname{Ric}^{0M}(V_{1})
-2\operatorname{Ric}^{0\mathcal V}(V_{1})
+\frac18
\left(
\|\operatorname{trace}T\|^{2}
+
\|\operatorname{trace}T^{*}\|^{2}
\right)
}.
\]
Thus equality is attained in the Chen--Ricci inequality.

We next verify the Hineva inequality. Since $r=2$, the terms
involving $r-2$ vanish. Hence the contribution corresponding to
$T$ is
\[
\begin{aligned}
\mathcal{T}
&=
\frac{r-1}{2r^{2}}
\left(
2\|\operatorname{trace}T\|^{2}
-r\|T\|^{2}
\right)\\
&=
\frac18
\left(
2(4a^{2})-2(2a^{2})
\right)\\
&=\frac12a^{2}.
\end{aligned}
\]
Similarly,
\[
\begin{aligned}
\mathcal{T}^{*}
&=
\frac{r-1}{2r^{2}}
\left(
2\|\operatorname{trace}T^{*}\|^{2}
-r\|T^{*}\|^{2}
\right)\\
&=
\frac18
\left(
2(4a^{2})-2(2a^{2})
\right)\\
&=\frac12a^{2}.
\end{aligned}
\]
Therefore,
\[
\mathcal{T}+\mathcal{T}^{*}=a^{2}.
\]
The Hineva inequality consequently reduces to
\[
\operatorname{Ric}^{M}_{\mathcal V}(V_{1})
\geq
0+2(0)-2(0)+a^{2}.
\]
Thus
\[
\operatorname{Ric}^{M}_{\mathcal V}(V_{1})
\geq a^{2}.
\]
Since
\[
\operatorname{Ric}^{M}_{\mathcal V}(V_{1})=a^{2},
\]
we have
\[
\boxed{
\operatorname{Ric}^{M}_{\mathcal V}(V_{1})
=a^{2}
}
\]
and equality is also attained in the Hineva inequality.

Therefore, this example simultaneously realizes equality in both the
Chen--Ricci upper bound and the Hineva lower bound along the vertical
distribution.
\end{example}

\begin{example}
Consider
\[
M=\mathbb{R}^{4},\qquad N=\mathbb{R}^{2},
\]
endowed with the standard Euclidean metrics
\[
g_{1}=dx^{2}+dy^{2}+dz^{2}+dw^{2},
\qquad
g_{2}=dx^{2}+dy^{2},
\]
and define the map
\[
F:M\longrightarrow N,\qquad
F(x,y,z,w)=(x,y).
\]
Let
\[
\mathcal{V}=\ker F_{*}
=\operatorname{span}\{V_{1},V_{2}\},
\qquad
\mathcal{H}=\mathcal{V}^{\perp}
=\operatorname{span}\{h_{1},h_{2}\},
\]
where
\[
V_{1}=\frac{\partial}{\partial z},\qquad
V_{2}=\frac{\partial}{\partial w},\qquad
h_{1}=\frac{\partial}{\partial x},\qquad
h_{2}=\frac{\partial}{\partial y}.
\]
Thus
\[
r=\dim\mathcal{V}=2,\qquad
s=\dim\mathcal{H}=2.
\]

Let $\nabla^{0}$ denote the Euclidean Levi--Civita connection on
$M$. Fix two distinct nonzero constants $a,b$ with
\[
a\neq b,
\]
and define a symmetric $(1,2)$-tensor $K$ by
\[
K_{V_{1}}V_{1}=ah_{1},\qquad
K_{V_{2}}V_{2}=bh_{1},
\]
\[
K_{V_{1}}h_{1}=aV_{1},\qquad
K_{V_{2}}h_{1}=bV_{2},
\]
and let all remaining components of $K$ be zero. Define two affine
connections on $M$ by
\[
\nabla_{X}Y=\nabla^{0}_{X}Y+K_{X}Y,
\qquad
\nabla^{*}_{X}Y=\nabla^{0}_{X}Y-K_{X}Y.
\]
Since $K$ is symmetric and
\[
g_{1}(K_{X}Y,Z)=g_{1}(Y,K_{X}Z),
\]
the connections $\nabla$ and $\nabla^{*}$ are torsion-free and
conjugate with respect to $g_{1}$. Hence
\[
(M,\nabla,g_{1})
\]
is a statistical manifold.

For horizontal basic vector fields $X,Y$, we have
\[
F_{*}(\nabla_{X}Y)
=
\nabla^{2}_{F_{*}X}F_{*}Y,
\]
where $\nabla^{2}$ is the Euclidean connection on $N$. Consequently,
\[
F:(M,\nabla,g_{1})
\longrightarrow
(N,\nabla^{2},g_{2})
\]
is a statistical submersion.

We now compute the fundamental tensors appearing in the
Chen--Ricci and Hineva inequalities. Since
\[
T_{V_{i}}V_{j}
=
\mathcal{H}\nabla_{V_{i}}V_{j},
\]
we obtain
\[
T_{V_{1}}V_{1}=ah_{1},\qquad
T_{V_{2}}V_{2}=bh_{1},
\]
and
\[
T_{V_{1}}V_{2}
=
T_{V_{2}}V_{1}
=
0.
\]
Therefore,
\[
\left(T^{1}_{ij}\right)
=
\begin{pmatrix}
a&0\\
0&b
\end{pmatrix},
\qquad
\left(T^{2}_{ij}\right)
=
\begin{pmatrix}
0&0\\
0&0
\end{pmatrix}.
\]

Similarly,
\[
T^{*}_{V_{i}}V_{j}
=
\mathcal{H}\nabla^{*}_{V_{i}}V_{j},
\]
and hence
\[
T^{*}_{V_{1}}V_{1}=-ah_{1},
\qquad
T^{*}_{V_{2}}V_{2}=-bh_{1},
\]
while
\[
T^{*}_{V_{1}}V_{2}
=
T^{*}_{V_{2}}V_{1}
=
0.
\]
Thus,
\[
\left(T^{*1}_{ij}\right)
=
\begin{pmatrix}
-a&0\\
0&-b
\end{pmatrix},
\qquad
\left(T^{*2}_{ij}\right)
=
\begin{pmatrix}
0&0\\
0&0
\end{pmatrix}.
\]

Consequently,
\[
T^{01}_{ij}
=
\frac12
\left(T^{1}_{ij}+T^{*1}_{ij}\right)
=
0,
\qquad
T^{02}_{ij}=0,
\]
and therefore
\[
T^{0}=0.
\]

Furthermore,
\[
\operatorname{trace}T
=
\sum_{i=1}^{2}T_{V_{i}}V_{i}
=
(a+b)h_{1},
\]
and hence
\[
\|\operatorname{trace}T\|^{2}
=
(a+b)^{2}.
\]
Similarly,
\[
\operatorname{trace}T^{*}
=
-(a+b)h_{1},
\]
so that
\[
\|\operatorname{trace}T^{*}\|^{2}
=
(a+b)^{2}.
\]
Moreover,
\[
\|T\|^{2}
=
\sum_{\alpha=1}^{2}
\sum_{i,j=1}^{2}
\left(T^{\alpha}_{ij}\right)^{2}
=
a^{2}+b^{2},
\]
and
\[
\|T^{*}\|^{2}
=
a^{2}+b^{2}.
\]

Since the fibres of $F$ are Euclidean planes, their intrinsic
curvature vanishes. Hence
\[
\operatorname{Ric}^{\mathcal V}(V_{1})=0.
\]
Also, with respect to the Levi--Civita connection $\nabla^{0}$,
\[
\operatorname{Ric}^{0M}(V_{1})=0,
\qquad
\operatorname{Ric}^{0\mathcal V}(V_{1})=0.
\]

On the other hand, using the connection $\nabla$, we have
\[
R(V_{1},V_{2})V_{2}
=
abV_{1}.
\]
Consequently,
\[
\operatorname{Ric}^{M}_{\mathcal V}(V_{1})
=
R(V_{1},V_{2},V_{2},V_{1})
=
ab.
\]

We now verify the Chen--Ricci inequality. By Theorem~3.1,
\[
\operatorname{Ric}^{M}_{\mathcal V}(V_{1})
\leq
\operatorname{Ric}^{\mathcal V}(V_{1})
+
2\operatorname{Ric}^{0M}(V_{1})
-
2\operatorname{Ric}^{0\mathcal V}(V_{1})
+
\frac18
\left(
\|\operatorname{trace}T\|^{2}
+
\|\operatorname{trace}T^{*}\|^{2}
\right).
\]
Substituting the above quantities, we obtain
\[
\begin{aligned}
\operatorname{Ric}^{M}_{\mathcal V}(V_{1})
&\leq
0+2(0)-2(0)
+
\frac18
\left(
(a+b)^{2}+(a+b)^{2}
\right)\\
&=
\frac14(a+b)^{2}.
\end{aligned}
\]
Since
\[
\operatorname{Ric}^{M}_{\mathcal V}(V_{1})=ab,
\]
we have
\[
ab\leq\frac14(a+b)^{2}.
\]
Indeed,
\[
\frac14(a+b)^{2}-ab
=
\frac14(a-b)^{2}.
\]
Since $a\neq b$,
\[
\frac14(a-b)^{2}>0.
\]
Therefore,
\[
\boxed{
\operatorname{Ric}^{M}_{\mathcal V}(V_{1})
<
\operatorname{Ric}^{\mathcal V}(V_{1})
+
2\operatorname{Ric}^{0M}(V_{1})
-
2\operatorname{Ric}^{0\mathcal V}(V_{1})
+
\frac18
\left(
\|\operatorname{trace}T\|^{2}
+
\|\operatorname{trace}T^{*}\|^{2}
\right)
}.
\]
Thus, the Chen--Ricci inequality is strict.

For a concrete choice, take
\[
a=1,\qquad b=2.
\]
Then
\[
\operatorname{Ric}^{M}_{\mathcal V}(V_{1})=2,
\]
whereas
\[
\frac18
\left(
\|\operatorname{trace}T\|^{2}
+
\|\operatorname{trace}T^{*}\|^{2}
\right)
=
\frac18(9+9)
=
\frac94.
\]
Consequently,
\[
\boxed{
2<\frac94,
}
\]
which explicitly demonstrates that equality is not attained in the
Chen--Ricci inequality.

We finally verify the Hineva inequality. Since $r=2$, the terms
containing $r-2$ vanish. The contribution corresponding to $T$ is
\[
\begin{aligned}
\mathcal{T}
&=
\frac{r-1}{2r^{2}}
\left(
2\|\operatorname{trace}T\|^{2}
-r\|T\|^{2}
\right)\\
&=
\frac18
\left(
2(a+b)^{2}
-
2(a^{2}+b^{2})
\right)\\
&=
\frac18(4ab)
=
\frac12ab.
\end{aligned}
\]
Similarly,
\[
\begin{aligned}
\mathcal{T}^{*}
&=
\frac{r-1}{2r^{2}}
\left(
2\|\operatorname{trace}T^{*}\|^{2}
-r\|T^{*}\|^{2}
\right)\\
&=
\frac18
\left(
2(a+b)^{2}
-
2(a^{2}+b^{2})
\right)\\
&=
\frac12ab.
\end{aligned}
\]
Therefore,
\[
\mathcal{T}+\mathcal{T}^{*}
=
ab.
\]
The Hineva inequality consequently becomes
\[
\operatorname{Ric}^{M}_{\mathcal V}(V_{1})
\geq
0+2(0)-2(0)+ab,
\]
that is,
\[
\operatorname{Ric}^{M}_{\mathcal V}(V_{1})
\geq ab.
\]
But
\[
\operatorname{Ric}^{M}_{\mathcal V}(V_{1})=ab.
\]
Hence equality is attained in the Hineva inequality:
\[
\boxed{
\operatorname{Ric}^{M}_{\mathcal V}(V_{1})=ab.
}
\]

Therefore, this example gives a statistical submersion for which the
Chen--Ricci inequality is strict, while the Hineva inequality is
attained as an equality. In particular, for $a=1$ and $b=2$,
\[
\boxed{
2<\frac94
}
\]
for the Chen--Ricci inequality, whereas
\[
\boxed{
\operatorname{Ric}^{M}_{\mathcal V}(V_{1})=2
}
\]
for the Hineva inequality.
\end{example}

\begin{example}
Consider
\[
M=\mathbb{R}^{5},\qquad N=\mathbb{R}^{2},
\]
endowed with the standard Euclidean metrics
\[
g_{1}
=
dx^{2}+dy^{2}+dz_{1}^{2}+dz_{2}^{2}+dz_{3}^{2},
\qquad
g_{2}=dx^{2}+dy^{2},
\]
and define the map
\[
F:M\longrightarrow N,\qquad
F(x,y,z_{1},z_{2},z_{3})=(x,y).
\]
Let
\[
\mathcal{V}=\ker F_{*}
=
\operatorname{span}\{V_{1},V_{2},V_{3}\},
\qquad
\mathcal{H}=\mathcal{V}^{\perp}
=
\operatorname{span}\{h_{1},h_{2}\},
\]
where
\[
V_{1}=\frac{\partial}{\partial z_{1}},
\qquad
V_{2}=\frac{\partial}{\partial z_{2}},
\qquad
V_{3}=\frac{\partial}{\partial z_{3}},
\]
and
\[
h_{1}=\frac{\partial}{\partial x},
\qquad
h_{2}=\frac{\partial}{\partial y}.
\]
Thus
\[
r=\dim\mathcal{V}=3,
\qquad
s=\dim\mathcal{H}=2.
\]

Let $\nabla^{0}$ denote the Euclidean Levi--Civita connection on
$M$. Define a symmetric $(1,2)$-tensor $K$ by
\[
K_{V_{1}}V_{1}=2h_{1},
\qquad
K_{V_{2}}V_{2}=h_{1},
\qquad
K_{V_{3}}V_{3}=0,
\]
\[
K_{V_{1}}h_{1}
=
K_{h_{1}}V_{1}
=
2V_{1},
\qquad
K_{V_{2}}h_{1}
=
K_{h_{1}}V_{2}
=
V_{2},
\qquad
K_{V_{3}}h_{1}
=
K_{h_{1}}V_{3}
=
0,
\]
and let all remaining components of $K$ be zero. Define two affine
connections on $M$ by
\[
\nabla_{X}Y=\nabla^{0}_{X}Y+K_{X}Y,
\qquad
\nabla^{*}_{X}Y=\nabla^{0}_{X}Y-K_{X}Y.
\]
Since $K$ is symmetric and
\[
g_{1}(K_{X}Y,Z)=g_{1}(Y,K_{X}Z),
\]
the connections $\nabla$ and $\nabla^{*}$ are torsion-free and
conjugate with respect to $g_{1}$. Hence
\[
(M,\nabla,g_{1})
\]
is a statistical manifold.

For horizontal basic vector fields $X,Y$, we have
\[
K_{X}Y=0
\]
and therefore
\[
F_{*}(\nabla_{X}Y)
=
F_{*}(\nabla^{0}_{X}Y)
=
\nabla^{2}_{F_{*}X}F_{*}Y,
\]
where $\nabla^{2}$ is the Euclidean connection on $N$. Consequently,
\[
F:(M,\nabla,g_{1})
\longrightarrow
(N,\nabla^{2},g_{2})
\]
is a statistical submersion.

We now compute the fundamental tensors. Since
\[
T_{V_{i}}V_{j}
=
\mathcal{H}\nabla_{V_{i}}V_{j},
\]
we obtain
\[
T_{V_{1}}V_{1}=2h_{1},
\qquad
T_{V_{2}}V_{2}=h_{1},
\qquad
T_{V_{3}}V_{3}=0,
\]
and
\[
T_{V_{i}}V_{j}=0,
\qquad i\neq j.
\]
Therefore,
\[
\left(T^{1}_{ij}\right)
=
\begin{pmatrix}
2&0&0\\
0&1&0\\
0&0&0
\end{pmatrix},
\qquad
\left(T^{2}_{ij}\right)
=
\begin{pmatrix}
0&0&0\\
0&0&0\\
0&0&0
\end{pmatrix}.
\]

Similarly,
\[
T^{*}_{V_{i}}V_{j}
=
\mathcal{H}\nabla^{*}_{V_{i}}V_{j},
\]
and hence
\[
T^{*}_{V_{1}}V_{1}=-2h_{1},
\qquad
T^{*}_{V_{2}}V_{2}=-h_{1},
\qquad
T^{*}_{V_{3}}V_{3}=0,
\]
while
\[
T^{*}_{V_{i}}V_{j}=0,
\qquad i\neq j.
\]
Thus,
\[
\left(T^{*1}_{ij}\right)
=
\begin{pmatrix}
-2&0&0\\
0&-1&0\\
0&0&0
\end{pmatrix},
\qquad
\left(T^{*2}_{ij}\right)
=
\begin{pmatrix}
0&0&0\\
0&0&0\\
0&0&0
\end{pmatrix}.
\]

Consequently,
\[
T^{01}_{ij}
=
\frac12
\left(
T^{1}_{ij}+T^{*1}_{ij}
\right)
=0,
\qquad
T^{02}_{ij}=0,
\]
and therefore
\[
T^{0}=0.
\]

Furthermore,
\[
\operatorname{trace}T
=
\sum_{i=1}^{3}T_{V_i}V_i
=
3h_{1},
\]
and hence
\[
\left\|\operatorname{trace}T\right\|^{2}=9.
\]
Similarly,
\[
\operatorname{trace}T^{*}
=
-3h_{1},
\qquad
\left\|\operatorname{trace}T^{*}\right\|^{2}=9.
\]
Moreover,
\[
\begin{aligned}
\|T\|^{2}
&=
\sum_{\alpha=1}^{2}
\sum_{i,j=1}^{3}
\left(T^{\alpha}_{ij}\right)^{2}\\
&=
2^{2}+1^{2}\\
&=5,
\end{aligned}
\]
and
\[
\|T^{*}\|^{2}=5.
\]

Since the fibres of $F$ are Euclidean $3$-planes, their intrinsic
curvature vanishes. Hence
\[
\operatorname{Ric}^{\mathcal V}(V_{1})=0.
\]
Also, with respect to the Levi--Civita connection $\nabla^{0}$,
\[
\operatorname{Ric}^{0M}(V_{1})=0,
\qquad
\operatorname{Ric}^{0\mathcal V}(V_{1})=0.
\]

On the other hand, using the connection $\nabla$, we obtain
\[
R(V_{1},V_{2})V_{2}
=
K_{V_{1}}(K_{V_{2}}V_{2})
-
K_{V_{2}}(K_{V_{1}}V_{2})
=
K_{V_{1}}h_{1}
=
2V_{1},
\]
and
\[
R(V_{1},V_{3})V_{3}=0.
\]
Consequently,
\[
\begin{aligned}
\operatorname{Ric}^{M}_{\mathcal V}(V_{1})
&=
R(V_{1},V_{2},V_{2},V_{1})
+
R(V_{1},V_{3},V_{3},V_{1})\\
&=2+0\\
&=2.
\end{aligned}
\]

We first verify the Chen--Ricci inequality. By Theorem~3.1,
\[
\operatorname{Ric}^{M}_{\mathcal V}(V_{1})
\leq
\operatorname{Ric}^{\mathcal V}(V_{1})
+
2\operatorname{Ric}^{0M}(V_{1})
-
2\operatorname{Ric}^{0\mathcal V}(V_{1})
+
\frac18
\left(
\|\operatorname{trace}T\|^{2}
+
\|\operatorname{trace}T^{*}\|^{2}
\right).
\]
Substituting the above quantities gives
\[
\begin{aligned}
\operatorname{Ric}^{M}_{\mathcal V}(V_{1})
&\leq
0+2(0)-2(0)
+\frac18(9+9)\\
&=\frac94.
\end{aligned}
\]
Since
\[
\operatorname{Ric}^{M}_{\mathcal V}(V_{1})=2,
\]
we obtain
\[
\boxed{
2<\frac94.
}
\]
Thus, the Chen--Ricci inequality is strict.

We next verify the Hineva inequality. Since $r=3$, we have
\[
\frac{r-1}{2r^{2}}
=
\frac{2}{18}
=
\frac19.
\]
Moreover,
\[
\frac{
r\|T\|^{2}
-
\|\operatorname{trace}T\|^{2}
}{r-1}
=
\frac{3(5)-9}{2}
=
3.
\]
Therefore,
\[
\sqrt{
\frac{
r\|T\|^{2}
-
\|\operatorname{trace}T\|^{2}
}{r-1}
}
=
\sqrt{3}.
\]
The contribution corresponding to $T$ is
\[
\begin{aligned}
\mathcal{T}
&=
\frac{r-1}{2r^{2}}
\left[
2\|\operatorname{trace}T\|^{2}
-r\|T\|^{2}
-(r-2)
\|\operatorname{trace}T\|
\sqrt{
\frac{
r\|T\|^{2}
-
\|\operatorname{trace}T\|^{2}
}{r-1}
}
\right]\\
&=
\frac19
\left[
2(9)-3(5)
-(1)(3)\sqrt{3}
\right]\\
&=
\frac19
\left(
18-15-3\sqrt{3}
\right)\\
&=
\frac{1-\sqrt{3}}{3}.
\end{aligned}
\]
Similarly,
\[
\mathcal{T}^{*}
=
\frac{1-\sqrt{3}}{3}.
\]
Hence
\[
\mathcal{T}+\mathcal{T}^{*}
=
\frac{2(1-\sqrt{3})}{3}.
\]

Therefore, the Hineva inequality becomes
\[
\operatorname{Ric}^{M}_{\mathcal V}(V_{1})
\geq
0+2(0)-2(0)
+
\frac{2(1-\sqrt{3})}{3}.
\]
Thus,
\[
\operatorname{Ric}^{M}_{\mathcal V}(V_{1})
\geq
\frac{2(1-\sqrt{3})}{3}.
\]
Since
\[
\operatorname{Ric}^{M}_{\mathcal V}(V_{1})=2
\]
and
\[
2>
\frac{2(1-\sqrt{3})}{3},
\]
we obtain
\[
\boxed{
\operatorname{Ric}^{M}_{\mathcal V}(V_{1})
>
\frac{2(1-\sqrt{3})}{3}.
}
\]
Hence, the Hineva inequality is also strict.

Therefore, this example provides a statistical submersion for which
equality is not attained in either inequality. More precisely,
\[
\boxed{
\operatorname{Ric}^{M}_{\mathcal V}(V_{1})
<
\frac18
\left(
\|\operatorname{trace}T\|^{2}
+
\|\operatorname{trace}T^{*}\|^{2}
\right)
}
\]
and
\[
\boxed{
\operatorname{Ric}^{M}_{\mathcal V}(V_{1})
>
\mathcal{T}+\mathcal{T}^{*}.
}
\]
Thus, the Chen--Ricci upper bound and the Hineva lower bound are both
strict in this example.
\end{example}

\section{Chen--Ricci Inequality along the Horizontal Distribution}

In this section, we establish a Chen--Ricci inequality for statistical
submersions along the horizontal distribution.

\begin{theorem}
Let
\[
F:(M,\nabla ,g_{1})\longrightarrow (N,\nabla ^2,g_{2})
\]
be a statistical submersion between two statistical manifolds. Then, for
every unit horizontal vector $h_{1}\in\mathcal H_{p}$,
\begin{align}
\operatorname{Ric}_{\mathcal H}^{M}(h_{1})
\leq{}&
\operatorname{Ric}^{\mathcal H}(h_{1})
+\frac14\|\operatorname{trace}\mathcal A\|^{2}
+\frac18\sum_{\alpha=1}^{r}\sum_{j=2}^{s}
(\mathcal A_{1j}^{*\alpha})^{2}
\nonumber\\
&-2\sum_{\alpha=1}^{r}\sum_{j=2}^{s}
\left(
\mathcal A_{1j}^{\alpha}
+\frac14\mathcal A_{1j}^{*\alpha}
\right)^{2}.
\label{eq-CR-SS}
\end{align}
Consequently,
\begin{equation}
\operatorname{Ric}_{\mathcal H}^{M}(h_{1})
\leq
\operatorname{Ric}^{\mathcal H}(h_{1})
+\frac14\|\operatorname{trace}\mathcal A\|^{2}
+\frac18\sum_{\alpha=1}^{r}\sum_{j=2}^{s}
(\mathcal A_{1j}^{*\alpha})^{2}.
\label{eq-CR-SS-1}
\end{equation}
Equality in \eqref{eq-CR-SS-1} holds if and only if
\[
\mathcal A_{11}^{\alpha}
=
\sum_{j=2}^{s}\mathcal A_{jj}^{\alpha},
\qquad
\mathcal A_{1j}^{\alpha}
=
-\frac14\mathcal A_{1j}^{*\alpha},
\]
for every $\alpha=1,\ldots,r$ and $j=2,\ldots,s$.
\end{theorem}

\begin{proof}
From \eqref{eq-Gauss-hor} and the definition of the horizontal Ricci
curvature \eqref{eq-Ric-hor}, we obtain
\begin{align}
\operatorname{Ric}_{\mathcal H}^{M}(h_{1})
={}&
\operatorname{Ric}^{\mathcal H}(h_{1})
+\sum_{\alpha=1}^{r}
\mathcal A_{11}^{\alpha}
\sum_{j=2}^{s}\mathcal A_{jj}^{\alpha}
\nonumber\\
&-2\sum_{\alpha=1}^{r}\sum_{j=2}^{s}
(\mathcal A_{1j}^{\alpha})^{2}
-\sum_{\alpha=1}^{r}\sum_{j=2}^{s}
\mathcal A_{1j}^{\alpha}\mathcal A_{1j}^{*\alpha}.
\label{eq-hor-1}
\end{align}
Using the identity
\[
-2a^{2}-ab
=
\frac18b^{2}
-2\left(a+\frac14b\right)^{2},
\]
equation \eqref{eq-hor-1} becomes
\begin{align}
\operatorname{Ric}_{\mathcal H}^{M}(h_{1})
={}&
\operatorname{Ric}^{\mathcal H}(h_{1})
+\sum_{\alpha=1}^{r}
\mathcal A_{11}^{\alpha}
\sum_{j=2}^{s}\mathcal A_{jj}^{\alpha}
\nonumber\\
&+\frac18\sum_{\alpha=1}^{r}\sum_{j=2}^{s}
(\mathcal A_{1j}^{*\alpha})^{2}
\nonumber\\
&-2\sum_{\alpha=1}^{r}\sum_{j=2}^{s}
\left(
\mathcal A_{1j}^{\alpha}
+\frac14\mathcal A_{1j}^{*\alpha}
\right)^{2}.
\label{eq-hor-3}
\end{align}
Next, applying the elementary inequality
\[
d_{1}d_{2}
\leq
\frac{(d_{1}+d_{2})^{2}}{4},
\]
we obtain
\begin{equation}
\mathcal A_{11}^{\alpha}
\sum_{j=2}^{s}\mathcal A_{jj}^{\alpha}
\leq
\frac14
\left(
\mathcal A_{11}^{\alpha}
+\sum_{j=2}^{s}\mathcal A_{jj}^{\alpha}
\right)^{2}.
\label{eq-hor-4}
\end{equation}
Substituting \eqref{eq-hor-4} into \eqref{eq-hor-3} and using
\eqref{eq-nrm-A}, we arrive at
\begin{align}
\operatorname{Ric}_{\mathcal H}^{M}(h_{1})
\leq{}&
\operatorname{Ric}^{\mathcal H}(h_{1})
+\frac14\|\operatorname{trace}\mathcal A\|^{2}
+\frac18\sum_{\alpha=1}^{r}\sum_{j=2}^{s}
(\mathcal A_{1j}^{*\alpha})^{2}
\nonumber\\
&-2\sum_{\alpha=1}^{r}\sum_{j=2}^{s}
\left(
\mathcal A_{1j}^{\alpha}
+\frac14\mathcal A_{1j}^{*\alpha}
\right)^{2},
\label{eq-hor-5}
\end{align}
which proves \eqref{eq-CR-SS}. Since the last term in \eqref{eq-hor-5} is non-positive, inequality
\eqref{eq-CR-SS-1} follows immediately.

Finally, equality in \eqref{eq-CR-SS-1} holds if and only if equality is
attained in \eqref{eq-hor-4} and the squared term in
\eqref{eq-hor-5} vanishes identically. Hence
\[
\mathcal A_{11}^{\alpha}
=
\sum_{j=2}^{s}\mathcal A_{jj}^{\alpha},
\qquad
\mathcal A_{1j}^{\alpha}
=
-\frac14\mathcal A_{1j}^{*\alpha},
\]
for every $\alpha=1,\ldots,r$ and $j=2,\ldots,s$. This completes the
proof.
\end{proof}

\section{Chen-Ricci inequality along mixed distribution}

In this section, we give an alternative proof of the Chen–Ricci inequality based on a simple and direct argument

\begin{theorem}\cite{SiddiquiChenSiddiqi2021}
Let $F:\left( M,\nabla ,g_{1}\right) \rightarrow \left( N,\nabla
^{2},g_{2}\right) $ be a statistical submersion between two statistical
manifolds. Then for any unit vector $V_{1}\in {\cal V}_{p}$, we have%
\begin{eqnarray*}
&&Ric_{{\cal V}}^{M}\left( V_{1}\right) +\sum_{i=1}^{s}Ric_{{\cal V}%
}^{M}\left( h_{i}\right) \leq Ric^{{\cal V}}\left( V_{1}\right)
+2Ric^{0M}\left( p\right) -2Ric^{0{\cal V}}\left( p\right) \\
&&+\frac{1}{8}\left( \left\Vert {\rm trace}T\right\Vert ^{2}+\left\Vert {\rm %
trace}T^{\ast }\right\Vert ^{2}\right) +\tau ^{{\cal H}}+\frac{1}{2}%
\left\Vert \sigma \right\Vert ^{2}-\frac{1}{2}g\left( {\cal A},{\cal A}%
^{\ast }\right) \\
&&-\left\Vert {\cal A}\right\Vert ^{2}-\delta \left( N\right) -\delta ^{\ast
}\left( N^{\ast }\right) -\delta \left( \sigma \right) -\delta ^{\ast
}\left( \sigma ^{\ast }\right) .
\end{eqnarray*}%
The equality case follows Theorem\ref{Theorem-CR-vert}.
\end{theorem}

\noindent {\bf Proof. }we have
the scalar curvature $\tau ^{M}$ of $M$ is%
\begin{equation}
\tau ^{M}=\sum_{1\leq i<j\leq r}R^{M}\left( V_{i},V_{j},V_{j},V_{i}\right)
+\sum_{1\leq i<j\leq s}R^{M}\left( h_{i},h_{j},h_{j},h_{i}\right)
+\sum_{i=1}^{s}\sum_{j=1}^{r}R^{M}\left( h_{i},V_{j},V_{j},h_{i}\right) .
\label{eq-mix-1}
\end{equation}%
We can rewrite (\ref{eq-mix-1}) as%
\begin{eqnarray}
&&\tau ^{M}-\sum_{2\leq i<j\leq r}R^{M}\left( V_{i},V_{j},V_{j},V_{i}\right)
\nonumber \\
&=&\sum_{j=2}^{r}R^{M}\left( V_{1},V_{j},V_{j},V_{1}\right) +\sum_{1\leq
i<j\leq s}R^{M}\left( h_{i},h_{j},h_{j},h_{i}\right)
+\sum_{i=1}^{s}\sum_{j=1}^{r}R^{M}\left( h_{i},V_{j},V_{j},h_{i}\right) .
\label{eq-mix-2}
\end{eqnarray}%
From (\ref{eq-abc}), we have%
\begin{eqnarray}
Ric_{{\cal V}}^{M}\left( V_{1}\right) &=&Ric^{{\cal V}}\left( V_{1}\right)
+2\sum_{\alpha =1}^{s}\sum_{j=2}^{r}\left( T_{1j}^{0\alpha }\right)
^{2}-2\sum_{\alpha =1}^{s}\left( T_{11}^{0\alpha
}\sum_{j=2}^{r}T_{jj}^{0\alpha }\right) +\frac{1}{2}\sum_{\alpha
=1}^{s}\left( T_{11}^{\alpha }\sum_{j=2}^{r}T_{jj}^{\alpha }\right) 
\nonumber \\
&&-\frac{1}{2}\sum_{\alpha =1}^{s}\sum_{j=2}^{r}\left( T_{1j}^{\alpha
}\right) ^{2}+\frac{1}{2}\sum_{\alpha =1}^{s}\left( T_{11}^{\ast \alpha
}\sum_{j=2}^{r}T_{jj}^{\ast \alpha }\right) -\frac{1}{2}\sum_{\alpha
=1}^{s}\sum_{j=2}^{r}\left( T_{1j}^{\ast \alpha }\right) ^{2}  \label{Ric-v}
\end{eqnarray}%
From (\ref{eq-Gauss-hor}), we obtain%
\begin{eqnarray}
\sum_{j=2}^{s}R^{M}\left( h_{1},h_{j},h_{j},h_{1}\right) &=&\sum_{j=2}^{s}R^{%
{\cal H}}\left( h_{1},h_{j},h_{j},h_{1}\right) -\sum_{j=2}^{s}g\left( {\cal A%
}_{h_{j}}h_{j},{\cal A}_{h_{1}}^{\ast }h_{1}\right)  \nonumber \\
&&+\sum_{j=2}^{s}g\left( \left( {\cal A}_{h_{1}}+{\cal A}_{h_{1}}^{\ast
}\right) h_{j},{\cal A}_{h_{j}}^{\ast }h_{1}\right) +\sum_{j=2}^{s}g\left( 
{\cal A}_{h_{1}}h_{j},{\cal A}_{h_{j}}^{\ast }h_{1}\right)
\label{eq-Ric-hor-Aij}
\end{eqnarray}%
Using (\ref{eq-Ric-hor}) and (\ref{eq-A-components}) in (\ref{eq-Ric-hor-Aij}), we
obtain%
\begin{equation}
\sum_{1\leq i<j\leq s}R^{M}\left( h_{i},h_{j},h_{j},h_{i}\right) =\tau ^{%
{\cal H}}+\frac{1}{2}\left\Vert \sigma \right\Vert ^{2}-\frac{1}{2}g\left( 
{\cal A},{\cal A}^{\ast }\right) -\left\Vert {\cal A}\right\Vert ^{2}
\label{Ric-H}
\end{equation}%
Using (\ref{Ric-v}), (\ref{Ric-H}) and (\ref{eq-Gauss-mix}) in (\ref{eq-mix-2}),
we obtain%
\begin{eqnarray}
&&\tau ^{M}-\sum_{2\leq i<j\leq r}R^{M}\left( V_{i},V_{j},V_{j},V_{i}\right)
=Ric^{{\cal V}}\left( V_{1}\right) +2\sum_{\alpha
=1}^{s}\sum_{j=2}^{r}\left( T_{1j}^{0\alpha }\right) ^{2}-2\sum_{\alpha
=1}^{s}\left( T_{11}^{0\alpha }\sum_{j=2}^{r}T_{jj}^{0\alpha }\right) 
\nonumber \\
&&+\frac{1}{2}\sum_{\alpha =1}^{s}\left( T_{11}^{\alpha
}\sum_{j=2}^{r}T_{jj}^{\alpha }\right) -\frac{1}{2}\sum_{\alpha
=1}^{s}\sum_{j=2}^{r}\left( T_{1j}^{\alpha }\right) ^{2}+\frac{1}{2}%
\sum_{\alpha =1}^{s}\left( T_{11}^{\ast \alpha }\sum_{j=2}^{r}T_{jj}^{\ast
\alpha }\right) -\frac{1}{2}\sum_{\alpha =1}^{s}\sum_{j=2}^{r}\left(
T_{1j}^{\ast \alpha }\right) ^{2}  \nonumber \\
&&+\tau ^{{\cal H}}+\frac{1}{2}\left\Vert \sigma \right\Vert ^{2}-\frac{1}{2}%
g\left( {\cal A},{\cal A}^{\ast }\right) -\left\Vert {\cal A}\right\Vert
^{2}-\delta \left( N\right) -\delta ^{\ast }\left( N^{\ast }\right) -\delta
\left( \sigma \right) -\delta ^{\ast }\left( \sigma ^{\ast }\right) .
\label{eq-mix-3}
\end{eqnarray}%
Thus, we have%
\begin{eqnarray}
&&\tau ^{M}-\sum_{2\leq i<j\leq r}R^{M}\left( V_{i},V_{j},V_{j},V_{i}\right)
\leq Ric^{{\cal V}}\left( V_{1}\right) +2\sum_{\alpha
=1}^{s}\sum_{j=2}^{r}\left( T_{1j}^{0\alpha }\right) ^{2}-2\sum_{\alpha
=1}^{s}\left( T_{11}^{0\alpha }\sum_{j=2}^{r}T_{jj}^{0\alpha }\right) 
\nonumber \\
&&+\frac{1}{2}\sum_{\alpha =1}^{s}\left( T_{11}^{\alpha
}\sum_{j=2}^{r}T_{jj}^{\alpha }\right) +\frac{1}{2}\sum_{\alpha
=1}^{s}\left( T_{11}^{\ast \alpha }\sum_{j=2}^{r}T_{jj}^{\ast \alpha
}\right) +\tau ^{{\cal H}}+\frac{1}{2}\left\Vert \sigma \right\Vert ^{2}-%
\frac{1}{2}g\left( {\cal A},{\cal A}^{\ast }\right) -\left\Vert {\cal A}%
\right\Vert ^{2}  \nonumber \\
&&-\delta \left( N\right) -\delta ^{\ast }\left( N^{\ast }\right) -\delta
\left( \sigma \right) -\delta ^{\ast }\left( \sigma ^{\ast }\right) .
\label{eq-mix-4}
\end{eqnarray}%
we now use the elementary inequality%
\[
d_{1}d_{2}\leq \frac{\left( d_{1}+d_{2}\right) ^{2}}{4} 
\]%
with equality if and only if $d_{1}=d_{2}$, Suppose $d_{1}=T_{11}^{\alpha }$
and $d_{2}=\sum_{j=2}^{r}T_{jj}^{\alpha }$, and similarly for the dual
components, we obtain%
\begin{eqnarray}
&&\tau ^{M}-\sum_{2\leq i<j\leq r}R^{M}\left( V_{i},V_{j},V_{j},V_{i}\right)
\leq Ric^{{\cal V}}\left( V_{1}\right) +2\sum_{\alpha
=1}^{s}\sum_{j=2}^{r}\left( T_{1j}^{0\alpha }\right) ^{2}-2\sum_{\alpha
=1}^{s}\left( T_{11}^{0\alpha }\sum_{j=2}^{r}T_{jj}^{0\alpha }\right) 
\nonumber \\
&&+\frac{1}{8}\sum_{\alpha =1}^{s}\left( T_{11}^{\alpha
}+\sum_{j=2}^{r}T_{jj}^{\alpha }\right) ^{2}+\frac{1}{8}\sum_{\alpha
=1}^{s}\left( T_{11}^{\ast \alpha }+\sum_{j=2}^{r}T_{jj}^{\ast \alpha
}\right) ^{2}+\tau ^{{\cal H}}+\frac{1}{2}\left\Vert \sigma \right\Vert ^{2}
\nonumber \\
&&-\frac{1}{2}g\left( {\cal A},{\cal A}^{\ast }\right) -\left\Vert {\cal A}%
\right\Vert ^{2}-\delta \left( N\right) -\delta ^{\ast }\left( N^{\ast
}\right) -\delta \left( \sigma \right) -\delta ^{\ast }\left( \sigma ^{\ast
}\right) .  \label{eq-mix-5}
\end{eqnarray}%
Using (\ref{eq-trace-T}) in (\ref{eq-mix-5}), we obtain%
\begin{eqnarray}
&&\tau ^{M}-\sum_{2\leq i<j\leq r}R^{M}\left( V_{i},V_{j},V_{j},V_{i}\right)
\leq Ric^{{\cal V}}\left( V_{1}\right) +2\sum_{\alpha
=1}^{s}\sum_{j=2}^{r}\left( T_{1j}^{0\alpha }\right) ^{2}  \nonumber \\
&&-2\sum_{\alpha =1}^{s}\left( T_{11}^{0\alpha
}\sum_{j=2}^{r}T_{jj}^{0\alpha }\right) +\frac{1}{8}\left( \left\Vert {\rm %
trace}T\right\Vert ^{2}+\left\Vert {\rm trace}T^{\ast }\right\Vert
^{2}\right) +\tau ^{{\cal H}}  \nonumber \\
&&+\frac{1}{2}\left\Vert \sigma \right\Vert ^{2}-\frac{1}{2}g\left( {\cal A},%
{\cal A}^{\ast }\right) -\left\Vert {\cal A}\right\Vert ^{2}-\delta \left(
N\right) -\delta ^{\ast }\left( N^{\ast }\right) -\delta \left( \sigma
\right) -\delta ^{\ast }\left( \sigma ^{\ast }\right) .  \label{eq-mix-6}
\end{eqnarray}%
From (\ref{eq-CR-ineq-4}), we have%
\begin{equation}
2Ric^{0M}\left( p\right) -2Ric^{0{\cal V}}\left( p\right) =2\sum_{\alpha
=1}^{s}\sum_{j=2}^{r}\left( T_{1j}^{0\alpha }\right) ^{2}-2\sum_{\alpha
=1}^{s}\left( T_{11}^{0\alpha }\sum_{j=2}^{r}T_{jj}^{0\alpha }\right)
\label{eq-mix-7}
\end{equation}%
Using (\ref{eq-mix-7}) in (\ref{eq-mix-6}), we obtain%
\begin{eqnarray}
&&\tau ^{M}-\sum_{2\leq i<j\leq r}R^{M}\left( V_{i},V_{j},V_{j},V_{i}\right)
\leq Ric^{{\cal V}}\left( V_{1}\right) +2Ric^{0M}\left( p\right) -2Ric^{0%
{\cal V}}\left( p\right)  \nonumber \\
&&+\frac{1}{8}\left( \left\Vert {\rm trace}T\right\Vert ^{2}+\left\Vert {\rm %
trace}T^{\ast }\right\Vert ^{2}\right) +\tau ^{{\cal H}}+\frac{1}{2}%
\left\Vert \sigma \right\Vert ^{2}-\frac{1}{2}g\left( {\cal A},{\cal A}%
^{\ast }\right)  \nonumber \\
&&-\left\Vert {\cal A}\right\Vert ^{2}-\delta \left( N\right) -\delta ^{\ast
}\left( N^{\ast }\right) -\delta \left( \sigma \right) -\delta ^{\ast
}\left( \sigma ^{\ast }\right) .  \label{eq-mix-7.4}
\end{eqnarray}%
Using (\ref{eq-Ric-vert}), (\ref{eq-Ric-hor}) and (\ref{eq-Ric-mix}) in (\ref%
{eq-mix-7.4}), we obtain%
\begin{eqnarray}
&&Ric_{{\cal V}}^{M}\left( V_{1}\right) +\sum_{i=1}^{s}Ric_{{\cal V}%
}^{M}\left( h_{i}\right) \leq Ric^{{\cal V}}\left( V_{1}\right)
+2Ric^{0M}\left( p\right) -2Ric^{0{\cal V}}\left( p\right)  \nonumber \\
&&+\frac{1}{8}\left( \left\Vert {\rm trace}T\right\Vert ^{2}+\left\Vert {\rm %
trace}T^{\ast }\right\Vert ^{2}\right) +\tau ^{{\cal H}}+\frac{1}{2}%
\left\Vert \sigma \right\Vert ^{2}-\frac{1}{2}g\left( {\cal A},{\cal A}%
^{\ast }\right)  \nonumber \\
&&-\left\Vert {\cal A}\right\Vert ^{2}-\delta \left( N\right) -\delta ^{\ast
}\left( N^{\ast }\right) -\delta \left( \sigma \right) -\delta ^{\ast
}\left( \sigma ^{\ast }\right) .  \label{eq-mix-8}
\end{eqnarray}%
The equality case follows Theorem\ref{Theorem-CR-vert}.

\subsection{Hineva inequality along mixed distribution}

In this section, we obtain Hineva inequality for statistical submersion
along mixed distribution.

\begin{theorem}
Let $F:\left( M,\nabla ,g_{1}\right) \rightarrow \left( N,\nabla
^{2},g_{2}\right) $ be a statistical submersion between two statistical
manifolds. Then for any unit vector $V_{1}\in {\cal V}_{p}$, we have%
\begin{eqnarray*}
&&Ric_{{\cal V}}^{M}\left( V_{1}\right) +\sum_{i=1}^{s}Ric_{{\cal V}%
}^{M}\left( h_{i}\right) \geq Ric^{{\cal V}}\left( V_{1}\right)
+2Ric^{0M}\left( p\right) -2Ric^{0{\cal V}}\left( p\right) \\
&&+\frac{r-1}{2r^{2}}\left( 2\left\Vert {\rm trace}T\right\Vert
^{2}-r\left\Vert T\right\Vert ^{2}-\left( r-2\right) \left\Vert {\rm trace}%
T\right\Vert \sqrt{\frac{r\left\Vert T\right\Vert ^{2}-\left\Vert {\rm trace}%
T\right\Vert ^{2}}{r-1}}\right) \\
&&+\frac{r-1}{2r^{2}}\left( 2\left\Vert {\rm trace}T^{\ast }\right\Vert
^{2}-r\left\Vert T^{\ast }\right\Vert ^{2}-\left( r-2\right) \left\Vert {\rm %
trace}T^{\ast }\right\Vert \sqrt{\frac{r\left\Vert T^{\ast }\right\Vert
^{2}-\left\Vert {\rm trace}T^{\ast }\right\Vert ^{2}}{r-1}}\right) \\
&&+\tau ^{{\cal H}}+\frac{1}{2}\left\Vert \sigma \right\Vert ^{2}-\frac{1}{2}%
g\left( {\cal A},{\cal A}^{\ast }\right) -\left\Vert {\cal A}\right\Vert
^{2}-\delta \left( N\right) -\delta ^{\ast }\left( N^{\ast }\right) -\delta
\left( \sigma \right) -\delta ^{\ast }\left( \sigma ^{\ast }\right) .
\end{eqnarray*}%
The equality case follows Theorem \ref{Theorem-Hineva-vert}.
\end{theorem}

\noindent {\bf Proof.} From (\ref{eq-mix-3}), we have%
\begin{eqnarray}
&&\tau ^{M}-\sum_{2\leq i<j\leq r}R^{M}\left( V_{i},V_{j},V_{j},V_{i}\right)
=Ric^{{\cal V}}\left( V_{1}\right) +2\sum_{j=2}^{r}\left( T_{1j}^{0\alpha
}\right) ^{2}-2T_{11}^{0\alpha }\sum_{j=2}^{r}T_{jj}^{0\alpha }  \nonumber \\
&&+\frac{1}{2}T_{11}^{\alpha }\sum_{j=2}^{r}T_{jj}^{\alpha }-\frac{1}{2}%
\sum_{j=2}^{r}\left( T_{1j}^{\alpha }\right) ^{2}+\frac{1}{2}T_{11}^{\ast
\alpha }\sum_{j=2}^{r}T_{jj}^{\ast \alpha }-\frac{1}{2}\sum_{j=2}^{r}\left(
T_{1j}^{\ast \alpha }\right) ^{2}+\tau ^{{\cal H}}  \nonumber \\
&&+\frac{1}{2}\left\Vert \sigma \right\Vert ^{2}-\frac{1}{2}g\left( {\cal A},%
{\cal A}^{\ast }\right) -\left\Vert {\cal A}\right\Vert ^{2}-\delta \left(
N\right) -\delta ^{\ast }\left( N^{\ast }\right) -\delta \left( \sigma
\right) -\delta ^{\ast }\left( \sigma ^{\ast }\right) .  \label{eq-Hin-mix-1}
\end{eqnarray}%
Using (\ref{eq-Lemma-Hineva-RM}) and (\ref{eq-Lemma-Hineva-RM-1}) in (\ref%
{eq-Hin-mix-1}), we obtain%
\begin{eqnarray}
&&\tau ^{M}-\sum_{2\leq i<j\leq r}R^{M}\left( V_{i},V_{j},V_{j},V_{i}\right)
\geq Ric^{{\cal V}}\left( V_{1}\right) +2\sum_{\alpha
=1}^{s}\sum_{j=2}^{r}\left( T_{1j}^{0\alpha }\right) ^{2}-2\sum_{\alpha
=1}^{s}\left( T_{11}^{0\alpha }\sum_{j=2}^{r}T_{jj}^{0\alpha }\right) 
\nonumber \\
&&+\frac{r-1}{2r^{2}}\left( 2\sum_{\alpha =1}^{s}\left(
\sum_{j=1}^{r}T_{jj}^{\alpha }\right) ^{2}-r\sum_{\alpha
=1}^{s}\sum_{i,j=1}^{r}\left( T_{ij}^{\alpha }\right) ^{2}-\left( r-2\right)
\sum_{\alpha =1}^{s}\left\vert \sum_{j=1}^{r}T_{jj}^{\alpha }\right\vert
\right.  \nonumber \\
&&\left. \sqrt{\frac{r\sum_{i,j=1}^{r}\left( T_{ij}^{\alpha }\right)
^{2}-\left( \sum_{j=1}^{r}T_{jj}^{\alpha }\right) ^{2}}{r-1}}\right) +\frac{%
r-1}{2r^{2}}\left( 2\sum_{\alpha =1}^{s}\left( \sum_{j=1}^{r}T_{jj}^{\ast
\alpha }\right) ^{2}\right.  \nonumber \\
&&\left. -r\sum_{\alpha =1}^{s}\sum_{i,j=1}^{r}\left( T_{ij}^{\ast \alpha
}\right) ^{2}-\left( r-2\right) \sum_{\alpha =1}^{s}\left\vert
\sum_{j=1}^{r}T_{jj}^{\ast \alpha }\right\vert \sqrt{\frac{%
r\sum_{i,j=1}^{r}\left( T_{ij}^{\ast \alpha }\right) ^{2}-\left(
\sum_{j=1}^{r}T_{jj}^{\ast \alpha }\right) ^{2}}{r-1}}\right)  \nonumber \\
&&+\tau ^{{\cal H}}+\frac{1}{2}\left\Vert \sigma \right\Vert ^{2}-\frac{1}{2}%
g\left( {\cal A},{\cal A}^{\ast }\right) -\left\Vert {\cal A}\right\Vert
^{2}-\delta \left( N\right) -\delta ^{\ast }\left( N^{\ast }\right) -\delta
\left( \sigma \right) -\delta ^{\ast }\left( \sigma ^{\ast }\right) .
\label{eq-Hin-mix-2}
\end{eqnarray}%
Using (\ref{eq-cauch-squartz-Hineva-RM}) and (\ref%
{eq-cauch-squartz-Hineva-RM-1}) in (\ref{eq-Hin-mix-2}), we obtain%
\begin{eqnarray}
&&\tau ^{M}-\sum_{2\leq i<j\leq r}R^{M}\left( V_{i},V_{j},V_{j},V_{i}\right)
-\sum_{2\leq i<j\leq s}R^{M}\left( h_{i},h_{j},h_{j},h_{i}\right)  \nonumber
\\
&\geq &Ric^{{\cal V}}\left( V_{1}\right) +2\sum_{\alpha
=1}^{s}\sum_{j=2}^{r}\left( T_{1j}^{0\alpha }\right) ^{2}-2\sum_{\alpha
=1}^{s}\left( T_{11}^{0\alpha }\sum_{j=2}^{r}T_{jj}^{0\alpha }\right) +\frac{%
r-1}{2r^{2}}\left( 2\left\Vert {\rm trace}T\right\Vert ^{2}-r\left\Vert
T\right\Vert ^{2}\right.  \nonumber \\
&&\left. -\left( r-2\right) \left\Vert {\rm trace}T\right\Vert \sqrt{\frac{%
r\left\Vert T\right\Vert ^{2}-\left\Vert {\rm trace}T\right\Vert ^{2}}{r-1}}%
\right) .+\frac{r-1}{2r^{2}}\left( 2\left\Vert {\rm trace}T^{\ast
}\right\Vert ^{2}-r\left\Vert T^{\ast }\right\Vert ^{2}\right.  \nonumber \\
&&\left. -\left( r-2\right) \left\Vert {\rm trace}T^{\ast }\right\Vert \sqrt{%
\frac{r\left\Vert T^{\ast }\right\Vert ^{2}-\left\Vert {\rm trace}T^{\ast
}\right\Vert ^{2}}{r-1}}\right) +\tau ^{{\cal H}}+\frac{1}{2}\left\Vert
\sigma \right\Vert ^{2}-\frac{1}{2}g\left( {\cal A},{\cal A}^{\ast }\right) 
\nonumber \\
&&-\left\Vert {\cal A}\right\Vert ^{2}-\delta \left( N\right) -\delta ^{\ast
}\left( N^{\ast }\right) -\delta \left( \sigma \right) -\delta ^{\ast
}\left( \sigma ^{\ast }\right) .  \label{eq-Hin-mix-3}
\end{eqnarray}%
Using (\ref{eq-mix-7}) in (\ref{eq-Hin-mix-3}), we obtain%
\begin{eqnarray}
&&\tau ^{M}-\sum_{2\leq i<j\leq r}R^{M}\left( V_{i},V_{j},V_{j},V_{i}\right)
\geq Ric^{{\cal V}}\left( V_{1}\right) +2Ric^{0M}\left( p\right) -2Ric^{0%
{\cal V}}\left( p\right)  \nonumber \\
&&+\frac{r-1}{2r^{2}}\left( 2\left\Vert {\rm trace}T\right\Vert
^{2}-r\left\Vert T\right\Vert ^{2}-\left( r-2\right) \left\Vert {\rm trace}%
T\right\Vert \sqrt{\frac{r\left\Vert T\right\Vert ^{2}-\left\Vert {\rm trace}%
T\right\Vert ^{2}}{r-1}}\right)  \nonumber \\
&&+\frac{r-1}{2r^{2}}\left( 2\left\Vert {\rm trace}T^{\ast }\right\Vert
^{2}-r\left\Vert T^{\ast }\right\Vert ^{2}-\left( r-2\right) \left\Vert {\rm %
trace}T^{\ast }\right\Vert \sqrt{\frac{r\left\Vert T^{\ast }\right\Vert
^{2}-\left\Vert {\rm trace}T^{\ast }\right\Vert ^{2}}{r-1}}\right)  \nonumber
\\
&&+\tau ^{{\cal H}}+\frac{1}{2}\left\Vert \sigma \right\Vert ^{2}-\frac{1}{2}%
g\left( {\cal A},{\cal A}^{\ast }\right) -\left\Vert {\cal A}\right\Vert
^{2}-\delta \left( N\right) -\delta ^{\ast }\left( N^{\ast }\right) -\delta
\left( \sigma \right) -\delta ^{\ast }\left( \sigma ^{\ast }\right) .
\label{eq-Hin-mix-4}
\end{eqnarray}%
Using (\ref{eq-Ric-vert}), (\ref{eq-Ric-hor}) and (\ref{eq-Ric-mix}) in (\ref%
{eq-Hin-mix-4}), we obtain%
\begin{eqnarray*}
&&Ric_{{\cal V}}^{M}\left( V_{1}\right) +\sum_{i=1}^{s}Ric_{{\cal V}%
}^{M}\left( h_{i}\right) \geq Ric^{{\cal V}}\left( V_{1}\right)
+2Ric^{0M}\left( p\right) -2Ric^{0{\cal V}}\left( p\right) \\
&&+\frac{r-1}{2r^{2}}\left( 2\left\Vert {\rm trace}T\right\Vert
^{2}-r\left\Vert T\right\Vert ^{2}-\left( r-2\right) \left\Vert {\rm trace}%
T\right\Vert \sqrt{\frac{r\left\Vert T\right\Vert ^{2}-\left\Vert {\rm trace}%
T\right\Vert ^{2}}{r-1}}\right) \\
&&+\frac{r-1}{2r^{2}}\left( 2\left\Vert {\rm trace}T^{\ast }\right\Vert
^{2}-r\left\Vert T^{\ast }\right\Vert ^{2}-\left( r-2\right) \left\Vert {\rm %
trace}T^{\ast }\right\Vert \sqrt{\frac{r\left\Vert T^{\ast }\right\Vert
^{2}-\left\Vert {\rm trace}T^{\ast }\right\Vert ^{2}}{r-1}}\right) \\
&&+\tau ^{{\cal H}}+\frac{1}{2}\left\Vert \sigma \right\Vert ^{2}-\frac{1}{2}%
g\left( {\cal A},{\cal A}^{\ast }\right) -\left\Vert {\cal A}\right\Vert
^{2}-\delta \left( N\right) -\delta ^{\ast }\left( N^{\ast }\right) -\delta
\left( \sigma \right) -\delta ^{\ast }\left( \sigma ^{\ast }\right) .
\end{eqnarray*}%
The equality case are similar to Theorem \ref{Theorem-Hineva-vert}.

\subsection{Simultaneous Chen--Ricci and Hineva inequalities for statistical
submersion along mixed distribution}

\begin{theorem}
Let $F:(N_{1},g_{1})\rightarrow (N_{2},g_{2})$ be a Riemannian submersion
between two Riemannian manifolds. Then%
\begin{eqnarray*}
&&\left. Ric_{{\cal V}}^{M}\left( V_{1}\right) +\sum_{i=1}^{s}Ric_{{\cal V}%
}^{M}\left( h_{i}\right) -2Ric^{0M}\left( p\right) +2Ric^{0{\cal V}}\left(
p\right) -\frac{1}{8}\left( \left\Vert {\rm trace}T\right\Vert
^{2}+\left\Vert {\rm trace}T^{\ast }\right\Vert ^{2}\right) -\tau ^{{\cal H}%
}-\varsigma \right.  \\
&&\left. \leq Ric^{{\cal V}}\left( V_{1}\right) \leq Ric_{{\cal V}%
}^{M}\left( V_{1}\right) +\sum_{i=1}^{s}Ric_{{\cal V}}^{M}\left(
h_{i}\right) -2Ric^{0M}\left( p\right) +2Ric^{0{\cal V}}\left( p\right)
-\tau ^{{\cal H}}-\zeta -\zeta ^{\ast }-\varsigma \right. 
\end{eqnarray*}%
The equality case follows Theorem\ref{Th-ineq-both-side-vert}.
\end{theorem}

\section{Conclusion and Future Research Directions}

In this paper, we have studied Ricci curvature inequalities for statistical
submersions by taking into account the vertical, horizontal, and mixed
distributions. In the vertical setting, we established a Chen--Ricci upper
bound and a Hineva-type lower bound involving the fundamental tensors
$T$ and $T^{*}$ and the intrinsic Ricci curvature of the fibres. Along the
horizontal distribution, we obtained a Chen--Ricci inequality in terms of
the tensors $A$ and $A^{*}$. Furthermore, by combining the vertical and
horizontal curvature relations, we derived corresponding estimates for the
mixed distribution. The equality cases were characterized through the
components and algebraic structures of the fundamental tensors and their
duals.

The examples presented in the paper demonstrate that the obtained
inequalities are not merely formal estimates. In particular, we exhibited
statistical submersions for which equality is attained and examples in
which the Chen--Ricci inequality is strict. Thus, the equality conditions
obtained in the theoretical results are genuinely realizable.

Several directions remain open for further investigation. It would be
interesting to establish analogous Ricci curvature inequalities for
special classes of statistical submersions, including statistical
submersions with invariant, anti-invariant, semi-invariant, and slant
structures. The obtained results may also be extended to statistical
submersions arising from statistical manifolds with almost Hermitian,
K\"ahler, Sasakian, Kenmotsu, cosymplectic, and quaternionic structures.
Another natural direction is to investigate Casorati, Wintgen, and
generalized $\delta$-invariant inequalities in the statistical submersion
setting. Finally, the interaction between the dual affine connections and
the geometry of the mixed distribution deserves further study, particularly
for curvature inequalities involving higher-order invariants.

\section*{Acknowledgements}

The author would like to express his sincere gratitude to the Department
of Mathematics, Banaras Hindu University, Varanasi, India, for providing
a supportive academic environment for this research.

\section*{Declarations}

\subsection*{Funding}

The author received financial support from the Council of Scientific and
Industrial Research (CSIR), New Delhi, India.

\subsection*{Conflict of Interest}

The author declares that there is no conflict of interest regarding the
publication of this paper.

\subsection*{Data Availability}

No datasets were generated or analyzed during the current study. The
results of this work are theoretical and do not involve any empirical
datasets.

\subsection*{Code Availability}

Not applicable.

\subsection*{Author Contributions}

The author was responsible for the conceptualization, methodology,
formal analysis, investigation, writing, and preparation of the manuscript.

\end{document}